\documentclass[11pt]{amsart}
\usepackage[left=1.2in, right=1.2in, top=1.2in, bottom=1.2in]{geometry}
\usepackage{layout}
\usepackage{amssymb, amsfonts,amsthm,amsmath,latexsym,mathrsfs,comment}
\usepackage{graphicx}
\usepackage{adjustbox}
\usepackage[all]{xy}
\usepackage{tikz}

\def\bt{\begin{thm}}
\def\et{\end{thm}}
\def\bl{\begin{lem}}
\def\el{\end{lem}}
\def\bd{\begin{defi}}
\def\ed{\end{defi}}
\def\bc{\begin{cor}}
\def\ec{\end{cor}}
\def\bp{\begin{proof}}
\def\ep{\end{proof}}
\def\br{\begin{rem}}
\def\er{\end{rem}}

\newtheorem{thm}{Theorem}[section]
\newtheorem{prop}[thm]{Proposition}
\newtheorem{lem}[thm]{Lemma}
\newtheorem{defn}[thm]{Definition}
\newtheorem{example}[thm]{Example}
\newtheorem{rem}[thm]{Remark}
\newtheorem{cor}[thm]{Corollary}

\numberwithin{equation}{section}

\newcommand{\A}{\mathcal{A}}
\newcommand{\Z}{\Bbb{Z}^m}
\newcommand{\R}{\Bbb{R}^m}
\newcommand{\C}{\Bbb{C}^m}
\newcommand{\Cdn}{\Bbb{C}^{d_N}}
\newcommand{\Pm}{\Bbb{P}^m}
\newcommand{\T}{(\Bbb{C}^*)^m}
\newcommand{\probn}{{\bf{P}_N}}
\newcommand{\p}{{\bf{P}}}
\newcommand{\pbn}{\sigma_N}
\newcommand{\U}{\mathscr{R}}
\newcommand{\la}{\langle}
\newcommand{\ra}{\rangle}

\newcommand{\bthm}{\begin{thm}}
\newcommand{\ethm}{\end{thm}}
\newcommand{\bstp}{\begin{stp}}
\newcommand{\estp}{\end{stp}}
\newcommand{\blem}{\begin{lemma}}
\newcommand{\elem}{\end{lemma}}
\newcommand{\bprop}{\begin{prop}}
\newcommand{\eprop}{\end{prop}}
\newcommand{\bpf}{\begin{pf}}
\newcommand{\epf}{\end{pf}}
\newcommand{\bdefn}{\begin{defn}}
\newcommand{\edefn}{\end{defn}}
\newcommand{\brk}{\begin{rmrk}}
\newcommand{\erk}{\end{rmrk}}
\newcommand{\bcrl}{\begin{crl}}
\newcommand{\ecrl}{\end{crl}}

\title{Zero distribution of random sparse polynomials}
\author{Turgay Bayraktar}
\date{\today}

\address{Mathematics Department, Syracuse University, NY, USA}
\email{tbayrakt@syr.edu}


\begin{document}

\begin{abstract}
We study asymptotic zero distribution of random Laurent polynomials whose support  are contained in dilates of a fixed integral polytope $P$ as their degree grow.  We consider a large class of probability distributions including the ones induced from i.i.d. random coefficients whose distribution law has bounded density with logarithmically decaying tails as well as moderate measures defined over the projectivized space of Laurent polynomials. We obtain a quantitative localized version of Bernstein-Kouchnirenko Theorem. 
\end{abstract}

\maketitle
\section{Introduction}
Recall that Newton polytope of a Laurent polynomial $f(z_1,\dots,z_m)\in\Bbb{C}[z_1^{\pm 1},\dots,z_m^{\pm 1}]$ is the convex hull (in $\R$) of the exponents of monomials in $f(z).$ It is well-known that for a system $(f_1,\dots,f_m)$ of Laurent polynomials in general position the common zeros is a discrete set in $\T:=(\Bbb{C}\setminus \{0\})^m$ and the number of simultaneous zeros of such a system is given by the mixed volume of Newton polytopes of $f_i's$ \cite{Bernstein,Ko}. In this work, we study asymptotic behavior of zeros of the systems of random Laurent polynomials with prescribed Newton polytope as their degree grow. More precisely, we consider Laurent polynomials whose support are contained in dilates $NP$ for a fixed integral polytope $P\subset\R$ with non-empty interior. Random Laurent polynomials with independent identically distributed (i.i.d.) coefficients whose distribution law is absolutely continuous with respect to Lebesgue measure and has logarithmically decaying tails arise as a special case. In particular, standard real and complex Gaussians are among the examples of such distributions. In another direction moderate measures defined on projectivized space of Laurent polynomials also fall into frame work of this paper.\\ \indent
 Computation of simultaneous zeros of deterministic as well as Gaussian systems of sparse polynomials has been studied by various authors (see eg. \cite{HSt,Rojas,MAR,DGS}) by using mostly methods of algebraic and toric geometry. In this work, we employ methods of pluripotential theory (cf. \cite{SZ1,DS3,BloomS, CM1,BloomL,B6}) which is extensively used in the dynamical study of holomorphic maps (see \cite{FS2} and references therein). Along the way, we develop a pluripotential theory for plurisubharmonic (psh for short) functions which are dominated by the support function of $P$ (up to a constant) in logarithmic coordinates on $\T$. We remark that the class of psh functions that we work with is a generalization of the Lelong class which corresponds here to the special case $P=\Sigma$ where $\Sigma$ is the standard unit simplex in $\R$. 
 For a weighted compact set $(K,q)$ i.e. a nonpluripolar compact set $K\subset \T$ and a continuous weight function $q:\T\to \Bbb{R}$, we define a weighted global extremal function $V_{P,K,q}$ on $\T.$ 
 Then for given integral polytopes $P_i$ with non-empty interior, we show that the mixed complex Monge-Amp\'ere measure $MA_{\Bbb{C}}(V_{P_1,K,q},\dots,V_{P_m,K,q})$ of the extremal functions $V_{P_i,K,q}$ is well defined on $\T$ and is of total mass equal to the mixed volume of $P_1,\dots,P_m.$  We use Bergman kernel asymptotics to prove that the normalized expected zero current along simultaneous zero set of independent random Laurent polynomials converges weakly to the external product $dd^cV_{P_1,K,q}\wedge\dots\wedge dd^cV_{P_k,K,q}$ in any codimension (Theorem \ref{main}). Moreover, if $P\subset \R_{\geq0},$ expected distribution of zeros has a self-averaging property in the sense that almost surely the normalized zero currents are asymptotic to $dd^cV_{P_1,K,q}\wedge\dots\wedge dd^cV_{P_k,K,q}$. In particular, almost surely number of zeros of $m$ independent Laurent polynomials $(f_1,\dots,f_m)$ in an open set $U\Subset \T$ is asymptotic to $N^mMA_{\Bbb{C}}(V_{P_1,K,q},\dots,V_{P_m,K,q})(U)$ (Theorem \ref{sa}). As a result, we obtain a quantitative localized version of Bernstein-Kouchnirenko theorem. In the last section, we obtain a generalization of the above results (Theorem \ref{main2}) for certain unbounded closed subsets $K\subset \T$ and certain weight functions $q.$ Recall that in the latter setting zero distribution of Gaussian Laurent polynomials is studied by Shiffman and Zelditch \cite{SZ1}. More precisely, the setting of \cite{SZ1} corresponds here to the special case $P\subset p\Sigma$ for some $p\in \Bbb{Z}_+,\ K=\T$ and $q(z)=\frac{p}{2}\log(1+\|z\|^2).$   \\ \indent
 For a Laurent polynomial $f$ the \textit{amoeba} $\mathscr{A}_f$ is by definition \cite{GKZ} the image of the zero locus of $f$ under the map $Log(z_1,\dots,z_m)=(\log|z_1|,\dots,\log|z_m|).$ Amoebas are useful tools in several areas such as complex analysis, real algebraic geometry and tropical algebra (see eg. \cite{PasR,FPT,Mi1,Mi3} and references therein). Complex plane curve amoebas were studied by Passare and Rullg{\aa}rd \cite{PasR} in which they proved that area of such amoebas is bounded by a constant times the volume of Newton polytope of $f.$ In certain cases, one can obtain asymptotic distribution of amoebas from our results.
 
\subsection{Statement of results}
Recall that a Laurent polynomial is of the form 
$$f(z)=\sum_J a_Jz^J\in\Bbb{C}[z_1^{\pm1},\dots,z_m^{\pm1}]$$
where $a_J\in\Bbb{C}$ and $z^J:=z_1^{j_1}\dots z_m^{j_m}.$ The set $S_f:=\{J\in\Z:a_J\not=0\}$ is called the \textit{support} of $f$ and convex hull of $S_f$ in $\R$ is called \textit{Newton polytope} of $f.$ For an integral polytope $P$ (i.e. convex hull of a finite subset of $\Z$), we denote the space of Laurent polynomials whose Newton polytope is contained in $P$ by
$$Poly(P):=\{f\in \Bbb{C}[z_1^{\pm1},\dots,z_m^{\pm1}]: S_f\subset P\}$$ Such polynomials are called \textit{sparse polynomials} in the literature. For each $N\in \Bbb{Z}_+$ we denote the $N$-dilate of $P$ by $NP.$ We let  $P_1,\dots, P_m$ denote integral polytopes with non-empty interior and we denote their mixed volume by $D:=MV_m(P_1,\dots,P_m).$ We assume that the mixed volume is normalized so that $MV_m(\Sigma):=MV_m(\Sigma,\dots,\Sigma)=1$ where $\Sigma:=\{t\in \R_{\geq0}: \sum_{j=1}^m t_j=1\}$ denotes the standard unit simplex in $\R.$\\ \indent  
 We are interested in asymptotic patterns of zero distribution of Laurent polynomial systems $(f_N^1,\dots,f_N^m)$ such that $S_{f^i_N}\subset NP_i$ as $N\to\infty.$ It follows from Bernstein-Kouchnirenko theorem \cite{Bernstein,Ko} that for systems in general position the set of common zeros are isolated points in $\T$ and the number of simultaneous roots of the system counting multiplicities is given by $DN^m.$  
      
 For a \textit{weighted compact set} $(K,q)$ i.e. a nonpluripolar compact set $K\subset \T$ and a continuous function $q:\T\to\Bbb{R},$ we define the weighted global extremal function 
 $$V_{P,K,q}:=\sup\{\psi\in Psh(\T): \psi(z)\leq \max_{J\in P}\log|z^J|+C_{\psi}\ \text{on}\ \T\ \text{and}\ \psi\leq q\ \text{on}\ K\}.$$ We remark that in the special case $P=\Sigma,$ the function $V_{\Sigma,K,q}$ coincides with the upper envelope of Lelong class of psh functions defined in \cite[Appendix B]{SaffTotik}.  It follows that $V_{P,K,q}$ is a locally bounded psh function on $\T$ and grows like the support function of $P$ in logarithmic coordinates (see section (\ref{global}) for details). By definition, a weighted compact set $(K,q)$ is regular if $V_{P,K,q}$ is continuous. Throughout this note we assume that $(K,q)$ is a regular weighted compact set. Unit polydisc and round sphere in $\C$ are among the examples of regular compact sets. \\ \indent 
For a measure $\tau$ supported in $K,$ we fix an orthonormal basis (ONB) $\{F^N_j\}_{j=1}^{d_N}$ for $Poly(NP)$ with respect to the inner product 
\begin{equation}\label{inner}
\la f,g\ra:=\int_{K}f(z)\overline{g(z)}e^{-2Nq(z)}d\tau(z). 
\end{equation} Then a Laurent polynomial $f_N$ can be written uniquely as
$$f_N=\sum_{j=1}^{d_N} a_jF^N_j$$ 
where $d_N=\dim(Poly(NP)).$  Throughout this note we assume that the Bergman functions associated with $Poly(P)$ $$B(\tau,q)(z):=\displaystyle\sup_{\|f\|_{L^2(e^{-2q}\tau)}=1}|f(z)|e^{-q(z)}$$  has sub-exponential growth, that is 
$$\sup_{z\in K} B(\tau,Nq)(z)=O(e^{N\epsilon})$$
 for all $\epsilon>0$ and $N\gg1.$  Such measures $\tau$ which always exist on regular weighted compact sets $(K,q)$ when $P\subset \R_{\geq0}$, are called Bernstein-Markov (BM) measures in the literature (see \S3.1 for details). 

 \subsection*{Randomization of $Poly(NP)$} We identify $Poly(NP)$ with $\Bbb{C}^{d_N}$ and endow it with a probability measure $\pbn.$ We remark that the probability space $(Poly(NP),\pbn)$ depends on the choice of ONB (i.e. the unitary identification $Poly(NP)\simeq \Cdn$ given by (\ref{inner})) unless $\sigma_N$ is the Gaussian induced by (\ref{inner}). However, asymptotic distribution of zeros is independent of the choice of this identification (cf Theorems \ref{main} and \ref{sa}). We also remark that our results apply in a quite general setting including random sparse polynomials with independent identically distributed (iid) coefficients whose distribution law has bounded density and logarithmically decaying tails (Proposition \ref{iid}) as well as  moderate measures (Proposition \ref{moderate}) supported on the unit sphere $S^{2d_N-1}$ with respect to the $L^2$ norm induced by (\ref{inner}).  

 It follows from Bertini's theorem that for generic systems $(f_N^1,\dots,f_N^k)$ of Laurent polynomials, their zero locuses are smooth and intersect transversely.  In particular,
$$Z_{f^1_N,\dots,f_N^k}:=\{z\in\T:f^1_N(z)=\dots=f_N^k(z)=0\}$$ 
is smooth and of codimension $k$ in $\T.$ 
We let $[Z_{f^1_N,\dots,f_N^k}]$ denote the current of integration along the zero set $Z_{f^1_N,\dots,f_N^k}.$ For generic systems $(f_N^1,\dots,f_N^k)$ the current $N^{-k}[Z_{f^1_N,\dots,f_N^k}]$ has finite mass on $\T$ bounded by the mixed volume $MV_m(P_1,\dots,P_k,\Sigma,\dots,\Sigma)$ (see Remark \ref{rem}) hence the \textit{expected zero current} 
$$\la\Bbb{E}[Z_{f^1_N,\dots,f_N^k}],\Theta\ra:=\int_{Poly(NP_1)\times\dots\times Poly(NP_k)}\la[Z_{f^1_N,\dots,f_N^k}],\Theta\ra d\sigma_N(f_N^1)\dots d\sigma_N(f_N^k)$$
 is well-defined on test forms $\Theta\in \mathcal{D}_{m-k,m-k}(\T).$  
\begin{thm}\label{main} Let $P_i\subset \R$ be an integral polytope with non-empty interior for each $i=1,\dots, m$ and $(K,q)$ be a regular weighted compact set. 
If
\begin{equation*}
\sup_{u\in S^{2d_N-1}}|\int_{\Bbb{C}^{d_N}}\log|\la a,u\ra|d\sigma_N(a)|=o(N)\ \text{as}\ N\to \infty \tag{$A1$}
\end{equation*}
then for each $1\leq k\leq m$
$$N^{-k}\Bbb{E}[Z_{f_N^1,\dots,f_N^k}] \to dd^c(V_{P_1,K,q})\wedge\dots \wedge dd^c(V_{P_k,K,q})$$ weakly on $\T$ as $N\to \infty.$ 
In particular, expected number of zeros
$$N^{-m}\Bbb{E}[\#\{z\in U: f_N^1(z)=\dots=f_N^m(z)=0\}] \to \int_U MA_{\Bbb{C}}(V_{P_1,K,q},\dots,V_{P_m,K,q})$$ as $N\to \infty$ for every smoothly bounded domain $U\subset\T.$
\end{thm}
Here $MA_{\Bbb{C}}(V_{P_1,K,q},\dots,V_{P_m,K,q})$ denotes the mixed complex Monge-Amp\'ere of the extremal functions $V_{P_1,K,q},\dots,V_{P_m,K,q}$ (see \S \ref{CMA} for details). 

In the special case $P\subset p\Sigma$ for some $p\in\Bbb{Z}_+,$ we can identify $Poly(NP)$ with a subspace $\Pi_{NP}$ of $H^0(\Bbb{P}^m,\mathcal{O}(pN))$ where $\mathcal{O}(1)\to \Bbb{P}^m$ denotes the hyperplane bundle on the complex projective space $\Bbb{P}^m.$ Then we consider the product space $\mathscr{P}=\prod_{N=1}^{\infty}\Pi_{NP}$ endowed with the product measure. Thus, elements of $\mathscr{P}$ are random sequences of global holomorphic sections of powers of $\mathcal{O}(p)$. Next, we obtain the following self averaging property of random zero currents.
\begin{thm}\label{sa}
Let $P_i\subset \R_{\geq0}$ be an integral polytope with non-empty interior  for each $i=1,\dots, m$ and $(K,q)$ be a regular weighted compact set.
If 
\begin{equation*}
\sum_{N=1}^{\infty}\sigma_N(a\in \Bbb{C}^{d_N}:\log\|a\|>N\epsilon)<\infty \ \text{for every}\ \epsilon>0 \tag{$A2$}
\end{equation*}
and for every $u\in S^{2d_N-1}$
\begin{equation*}
\sum_{N=1}^{\infty}\sigma_N(a\in \Bbb{C}^{d_N}:\log|\la a,u\ra|<-N t)<\infty \ \text{for every}\ t>0 \tag{$A3$}
\end{equation*}
  then for each $1\leq k\leq m$ almost surely
$$N^{-k}[Z_{f_N^1,\dots,f_N^k}] \to dd^c(V_{P_1,K,q})\wedge\dots \wedge dd^c(V_{P_k,K,q})$$ weakly on $\T$ as $N\to\infty.$
\end{thm}
In particular, when $k=m,$ it follows from Proposition \ref{mixma} that the total mass
$$\int_{\T}MA_{\Bbb{C}}(V_{P_1,K,q},\dots,V_{P_m,K,q})=MV_m(P_1,\dots,P_m).$$ 
Hence, almost surely the number of zeros in a domain $U\subset\T$ of $m$ independent random Laurent polynomials is asymptotic to $N^mMA_{\Bbb{C}}(V_{P_1,K,q},\dots,V_{P_m,K,q})(U).$ Thus, Theorem \ref{sa} gives a quantitative localized version of the Bernstein-Kouchnirenko theorem.
\subsection{Comparison with the results in the literature} 
Recall that a random Kac polynomial is of the form
$$f_N(z)=\sum_{j=0}^Na_jz^j$$ where coefficients $a_j$ are independent complex Gaussian random variables of mean zero and variance one. A classical result due to Kac and Hammersley \cite{Kac,Ham} asserts that normalized zeros of Kac random polynomials of large degree tend to accumulate on the unit circle $S^1=\{|z|=1\}.$ This ensemble of random polynomials has been extensively studied (see eg. \cite{LO,HN,SV,IZ}). Recently, Ibragimov and Zaporozhets \cite{IZ} proved that 
$$\Bbb{E}[\log(1+|a_j|)]<\infty$$ is a necessary and sufficient condition for zeros of random Kac polynomials to accumulate near the unit circle (see also the recent work \cite{TaoVu2} on local universality of zeros). In \cite{SZ3}, Shiffman and Zelditch remarked that it was an implicit choice of an inner product (see (\ref{inner})) that produced this concentration of zeros of Kac polynomials around the unit circle $S^1.$ More generally they proved that for a simply connected domain $\Omega\Subset\C$ with real analytic boundary $\partial\Omega$ and a fixed ONB $\{F_j^{N}\}_{j=1}^{n+1},$ zeros of random polynomials with i.i.d standard complex Gaussian coefficients $$f_N(z)=\sum_{j=1}^{N+1}a_jF^{N}_j(z)$$ concentrate near the boundary $\partial\Omega$ as $N\to \infty.$

\par Asymptotic zero distribution of multivariate random polynomials has been studied by several authors (see eg. \cite{SZ,SZ1,BloomS,Shiffman,DS3,BloomL,B6} and references therein). In particular, if the random coefficients $a_J$ in $f^i_N$ are i.i.d. standard complex Gaussian then we recover  \cite[Theorem 3.1]{BloomS})  (see also \cite[Theorem 7.3]{BloomL} and \cite[Theorem 1.2]{B6} for more general distributions).  On the other hand, Dinh and Sibony \cite{DS3} studied equidistribution problem by using formalism of meromorphic transforms. They considered moderate measures on the pojectivized space $\Bbb{P}Ploy(N\Sigma)$ which arise here as a special case. Recall that  Monge-Amp\`ere measure of a H\"{o}lder continuous qpsh function is among the examples of moderate measures (see \cite{DNS} for details).\\ \indent 
Theorem \ref{main} and \ref{sa} can be also considered as a global universality results in the sense that they extend some earlier known results for the Gaussian distributions to setting of distributions that has logarithmically decaying tails. For instance, letting $K=(S^1)^m$ the real torus and $q(z)\equiv0,$ we see that the monomials $\{z^J\}_{J\in NP\cap\Z}$ form an ONB for $Poly(NP)$ with respect to the normalized Lebesgue measure on the real torus. Moreover, endowing $Poly(NP)$ with complex (or real) Gaussian distribution with mean zero and a (positive definite and diagonal) variance matrix $C$ for each $1\leq k\leq m$ we observe that
$$N^{-k}\Bbb{E}[Z_{f_N^1,\dots,f^k_N}]=\omega_{NP_1}\wedge \dots \wedge \omega_{NP_k}$$ where $\omega_{NP_i}=\frac12dd^c\sum_{J\in NP_i\cap\Z}\log|z^J|^2$ is a K\"ahler form for sufficiently large $N$ and we obtain \cite[Theorem 2]{MAR}.  Then Example \ref{torus2} 
together with Theorem \ref{main} yields $$N^{-k}\Bbb{E}[Z_{f_N^1,\dots,f^m_N}]\to \frac{MV_m(P_1,\dots,P_m)}{(2\pi)^m}d\theta_1\dots d\theta_m\ \text{weakly as}\ N\to\infty$$
hence, we recover \cite[Theorem 1.8]{DGS}.

\par Next, we provide the following example to illustrate the impact of the choice of $(P,K)$ on zero distribution:
\begin{example}
Let $P=Conv((0,0),(0,1),(1,1),(T,0))\subset \Bbb{R}^2$ 
$$
\begin{tikzpicture} 
    \draw[fill=gray!50!white] plot[smooth,samples=100,domain=0:1](\x,{0}) -- 
    plot[smooth,samples=100,domain=1:0] (\x,{1});
\draw[fill=gray!50!white] plot[smooth,samples=100,domain=1:5](\x,{-0.25*\x+1.25}) -- 
    plot[smooth,samples=100,domain=1:0] (\x,{0});
    \node at (1,.5) {$P$};
     \draw[->] (-0.2,0) -- (6,0) node[right] {$x$};
    \draw[->] (0,-0.2) -- (0,2.5) node[above] {$y$};
    \draw (0,1) -- (1,1);
    \draw (1,1) -- (5,0);
   \fill (0,0) circle (2pt); 
   \fill (5,0) circle (2pt) node[above] {(T,0)};
    \fill (0,1) circle (2pt) node[left] {(0,1)};
     \fill (1,1) circle (2pt) node[above] {(1,1)};
      \end{tikzpicture}$$
where $T\geq 2$ is an integer and $K=S^3$ is the unit sphere in $\Bbb{C}^2.$ Then taking $q\equiv0$ we see that
$$c_Jz^J:=(\frac{(j_1+j_2+1)!}{j_1!j_2!})^{\frac12}z_1^{j_1}z_2^{j_2}\ \text{for}\ J=(j_1,j_2)\in NP$$
form an ONB for $Poly(NP)$ with respect to the inner product induced from $L^2(\sigma)$ where $\sigma$ is the probability surface area measure on $S^3.$ Then a random sparse polynomial is of the form 
\begin{equation}\label{rep} f_N(z)=\sum_{J\in NP}a_Jc_Jz^J.
\end{equation}
and by Theorem \ref{sa} almost surely
$$N^{-2}\sum_{\zeta\in Z_{f_N^1,f_N^2}}\delta_{\zeta}\to MA_{\Bbb{C}}(V_{P,K}).$$  
weakly as $N\to \infty$ where the measure $MA_{\Bbb{C}}(V_{P,K})$ is the complex Monge-Amp\'ere of the unweighted (i.e. $q\equiv0$) global extremal function $V_{P,K}$. By Proposition \ref{support}  the measure $MA_{\Bbb{C}}(V_{P,K})$ is supported in $S^3$. However, unlike the case $P=\Sigma$ the mass of $MA_{\Bbb{C}}(V_{P,K})$ is not uniformly distributed on $S^3$ (see Figures \ref{pic1} and \ref{pic2} below).\\ \indent 
 Figures \ref{pic1} and \ref{pic2} illustrate zero distribution of independent system of two random polynomials of the form (\ref{rep}) whose coefficients are complex i.i.d. standard Gaussian respectively Pareto-distributed with $T=5$ and $N=10$. 
\begin{figure}[ht]\centering
   \begin{minipage}{0.45\linewidth}
     \frame{\includegraphics[width=1\textwidth]{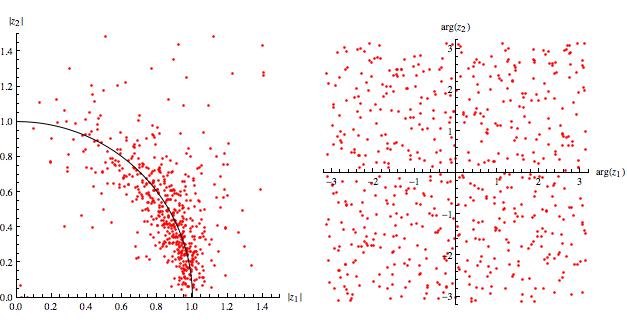}}
     \caption{Standard Gaussian }\label{pic1}
   \end{minipage}
   \quad
   \begin{minipage}{0.45\linewidth}
     \frame{\includegraphics[width=1\textwidth]{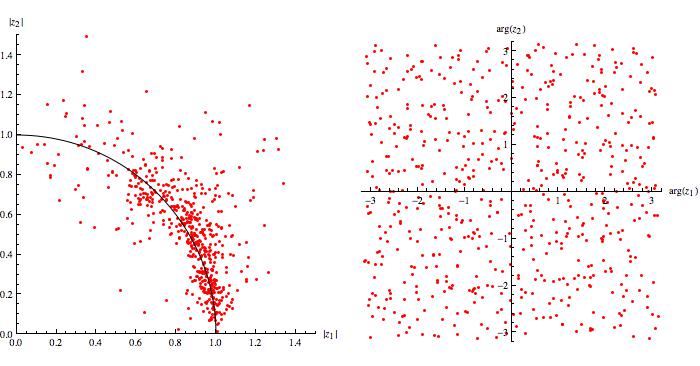}}
     \caption{Pareto distribution with $\p\{|a|>R\}\sim R^{-3}$}\label{pic2}
   \end{minipage}
   \end{figure}
\end{example}
In the last part of this work, we obtain a generalization of Theorem \ref{main} for certain unbounded closed sets $K\subset \T$ and weakly admissible weight functions $q$ (see $\S5$ for details):
\begin{thm}\label{main2}
Let $P_i\subset \R_{\geq0}$ be an integral polytope with non-empty interior and $(K,q_i)$ be a regular weighted  closed set with $q_i:\T\to\Bbb{R}$ be weakly admissible continuous weight function for each $i=1,\dots,k\leq m$. Assume that conditions $(A1),(A2)$ and $(A3)$ hold then
$$N^{-k}\Bbb{E}[Z_{f^1_N,\dots,f^k_N}] \to dd^c(V_{P_1,K,q_1})\wedge\dots \wedge dd^c(V_{P_k,K,q_k})$$ weakly as $N\to \infty.$ Moreover, almost surely 
$$N^{-k}[Z_{f_N^1,\dots,f_N^k}]\to dd^c(V_{P_1,K,q_1})\wedge\dots \wedge dd^c(V_{P_k,K,q_k})$$ weakly on $\T$  as $N\to \infty.$
\end{thm}

In the special case, $P_i\subset p\Sigma$ for some $p\in\Bbb{Z}_+$ and $K=\T$ together with $q(z)=\frac{p}{2}\log(1+\|z\|^2)$  zero distribution of random Laurent polynomials with i.i.d. standard complex Gaussian coefficients is studied by Shiffman and Zelditch \cite{SZ1,Shiffman}. It follows from \cite[Theorem 4.1]{SZ1} that $V_{K,P_i,q}$ is continuous on $\T,$ in particular $(K,q)$ is a regular weighted set (see Example \ref{SZex} for details). Hence, Theorem \ref{main2} applies in this setting and we recover \cite[Theorem 1.4]{SZ1} and \cite[Theorem 1.5]{Shiffman}. Specializing further, if $P:=P_1=\dots=P_m$ by Proposition \ref{support} we see that asymptoticly zeros of random polynomials concentrate in the region $\mathcal{A}_P:=\mu_p^{-1}(P^{\circ})$ which is called \textit{classically allowed region} in \cite{SZ1}, where  
$$\mu_p:\T\to\R$$
$$\mu_p(z)=(\frac{p|z_1|^2}{1+\|z\|^2},\dots,\frac{p|z_m|^2}{1+\|z\|^2}).$$

\begin{example}
Let $P=Conv((0,0),(0,1),(1,1),(1,0))\subset \Bbb{R}^2$ be the unit square.    
$$
\begin{tikzpicture} 
    \draw[fill=gray!50!white] plot[smooth,samples=100,domain=0:1](\x,{0}) -- 
    plot[smooth,samples=100,domain=1:0] (\x,{1});
    \node at (.5,.5) {$P$};
     \draw[->] (-0.2,0) -- (3,0) node[right] {$x$};
    \draw[->] (0,-0.2) -- (0,3) node[above] {$y$};
    \draw (0,1) -- (1,1);
    \draw (1,1) -- (1,0);
    \draw (2,0)--(0,2);
\node[below] at (2,0) {2};
    \node[below] at (1,0) {1};
    \node[left] at (0,2) {2};
    \node[left] at (0,1) {1};
     \end{tikzpicture}$$
We also let $K=(\Bbb{C}^*)^2$ and $q(z)=\log(1+\|z\|^2)$ (i.e. $p=2$). It follows form \cite[Example 1]{SZ1} that the classically allowed region is given by $$\mathcal{A}_P=\{(z_1,z_2)\in\T: |z_1|^2-1<|z_2|^2<|z_1|^2+1\}$$ and 
\begin{equation}\label{cform}
V_{P,K,q} (z_1,z_2)= \begin{cases}
\log(1+\|z\|^2) & \text{for}\ z\in \mathcal{A}_P\\
\frac12\log|z_2|^2+\frac12\log(1+|z_1|^2)+\log2  & \text{for}\ |z_2|^2\geq |z_1|^2+1\\
\frac12\log|z_1|^2+\frac12\log(1+|z_2|^2)+\log2  & \text{for}\ |z_1|^2\geq |z_2|^2+1
\end{cases}
\end{equation}
Hence, $(K,q)$ is a regular weighted closed set and Theorem \ref{sa} applies. Moreover, 
$$c_Jz^J:=(\frac{(N+2)!}{2!(N-|J|)!j_1!\dots j_2!})^{\frac12}z_1^{j_1}z_2^{j_2}$$ form an ONB for $Poly(NP)$ with respect to the inner product
\begin{eqnarray*}
\la f,g\ra: &=&\int_{(\Bbb{C}^*)^2}f(z)\overline{g(z)}e^{-2Nq(z)}\omega_{FS}^2 \\ &=& \int_{(\Bbb{C}^*)^2}f(z)\overline{g(z)} \frac{2}{\pi^2(1+\|z\|^2)^{2N+3}}dz. 
\end{eqnarray*}
Thus a random polynomial in the present setting is of the form
\begin{equation}\label{formm}
f_N(z)=\sum_{J\in NP} a_Jc_Jz^J 
\end{equation}
and by Theorem \ref{main2} almost surely 
$$N^{-2}\sum_{\zeta\in Z_{f_N^1,f_N^2}}\delta_{\zeta}\to 1_{\mathcal{A}}\frac{2}{\pi^2(1+\|z\|^2)^3}dz. $$
\end{example}

\subsection{Connection with toric varieties}
Recall that an integral polytope $P\subset\R$ is called Delzant if a neighborhood of any vertex of $P$ is $SL(m,\Bbb{Z})$ equivalent to $\{x_i\geq0:i=1,\dots,m\}\subset\R.$ A theorem of Delzant asserts that if $P$ is an integral Delzant polytope then one can construct a toric variety  $X_P$ which is a projective manifold and an ample line bundle $L\to X_P$ such that 
$\frac12dd^c\sum_{J\in NP\cap\Z}\log|z^J|^2$ is a K\"ahler metric on $\T$ and it extends to a smooth global K\"ahler metric on the toric variety $X_P$ for sufficiently large $N.$ Moreover, the space of global holomorphic sections $H^0(X_P,L^{\otimes N})$ can be identified with $Poly(NP).$ 
In this setting, the asymptotic distribution of zeros was obtained in \cite[Theorem 1.1]{B6} (see also \cite{SZ} for the Gaussian setting). 

\section{Preliminaries}
\subsection{Lattice points, polytopes and convex analysis}
In what follows $\R_+$ (respectively $\R_{\geq0}$ denotes the set of points in the real Euclidean space with positive (respectively non negative) coordinates. By an integral polytope we mean convex hull $Conv(\A)$ in $\R$ of a non-empty finite set $\A\subset\Z.$ We let $\Sigma$ denote the standard unit simplex that is $\Sigma=Conv(0,e_1,\dots,e_m)$ where $e_i$ denote the standard basis elements in $\Z.$ For two non-empty convex sets $P_1,P_2$ we denote their Minkowski sum by 
$$P_1+P_2:=\{x_1+x_2:x_1\in P_1,x_2\in P_2\}.$$
In the present section, we let $P\subset \R$ be a convex body i.e. a compact convex set with non-empty interior $P^{\circ}$. Let $Vol_m$ denote the volume of a subset of $\R$ with respect to Lebesgue measure which is normalized such that $Vol_m(\Sigma)=\frac{1}{m!}.$ \\ \indent
A theorem by Minkowski and Steiner asserts that $Vol_m(N_1P_1 + \dots + N_kP_k)$ is a
homogeneous polynomial of degree $m$ in the variables $N_1, \dots , N_k \in \Bbb{Z}_+$ (see for instance \cite[\S 4]{CLO} for details). 
In the special case $k=m,$ the coefficient of the monomial $N_1\cdots N_m$ in the homogenous expansion of $Vol_m(N_1P_1+\dots+N_mP_m)$ is called \textit{mixed volume} of $P_1,\dots,P_m$ and denoted by 
$MV_m(P_1,\dots,P_m).$  
One can compute the mixed volume of convex sets $P_1,\dots,P_m$ by means of polarization formula
$$MV_m(P_1,\dots,P_m)=\sum_{k=1}^m\sum_{1\leq j_1\leq\dots\leq j_k\leq m}(-1)^{m-k}Vol_m(P_{j_1}+\dots+P_{j_k}) .$$In particular, if $P=P_1=\dots=P_m$ then 
$$MV_m(P):=MV_m(P,\dots,P)=m!Vol_m(P).$$
In the special case, $MV_m(\Sigma)=1.$

We denote the \textit{support function} of a convex body $P$ by
$\varphi_P:\R\to\Bbb{R}$
$$\varphi_P(x)=\sup_{p\in P}\langle x,p\rangle$$
which is a one-homogenous convex function. We let $d\varphi_{|x}$ denote the \textit{sub-gradient} of $\varphi$ at $x\in \R.$ Recall that $d\varphi_{|x}$ is a closed convex set in $\R$ defined by 
$$d\varphi_{|x}:=\{p\in\R: \varphi(y)\geq \varphi(x)+\langle p,y-x\rangle \ \text{for every}\ y\in\R\}.$$  
We remark that if $\varphi$ is differentiable at $x$ then $d\varphi_{|x}$ is a point and coincides with $\nabla\varphi(x)$. In the sequel we let $d\varphi(E)$ denote the image of $E\subset\R$ under the sub-gradient. 
\subsubsection{Real Monge-Amp\'ere of a convex function}
Following \cite{BAR}, we define \textit{real Monge-Amp\'ere} (or Monge-Amp\'ere in the sense of Aleksandrov) of a finite convex function $\varphi$ by
\begin{equation}\label{defn}
MA_{\Bbb{R}}(\varphi)(E):=m!\int_{d\varphi(E)}dVol_m
\end{equation}
where $E\subset\R$ is a Borel set. The role of normalization constant $m!$ will be explained in (\ref{CMA}). If $\varphi\in \mathcal{C}^2(\R)$ then its real Monge-Amp\'ere coincides with its Hessian that is 
\begin{eqnarray}
MA_{\Bbb{R}}(\varphi)(E)
&=&m!\int_E\det(\frac{\partial^2\varphi}{\partial x_ix_j})dx \label{eq}.
\end{eqnarray} 
Moreover, for a convex function $\varphi\in \mathcal{C}^2(\R)$ one can also define the \textit{real Monge-Amp\'ere} as
$$\mathcal{MA}_{\Bbb{R}}(\varphi):=d(\varphi_{x_1})\wedge\dots \wedge d(\varphi_{x_m})$$ where $\varphi_{x_i}:=\frac{\partial\varphi}{\partial x_i}.$
In fact, endowing the cone of convex functions with the topology of  locally uniform convergence and the space of measures on $\R$ by topology of weak converge it follows from \cite{BAR} that the operator $\mathcal{MA}_{\Bbb{R}}$ extends as a continuous symmetric  multilinear operator and the equality 
$$MA_{\Bbb{R}}(\varphi)=\mathcal{MA}_{\Bbb{R}}(\varphi)$$
 remains valid for merely convex functions $\varphi.$ 
Finally,  following \cite{PasR} one can define \textit{mixed real Monge-Amp\'ere} of convex functions $\varphi_1,\dots,\varphi_m$ by means of the polarization formula 

\begin{equation}\label{rpol}
MA_{\Bbb{R}}(\varphi_1,\dots,\varphi_m):=\frac{1}{m!}\sum_{k=1}^m\sum_{1\leq j_1\leq\dots\leq j_k\leq m}(-1)^{m-k}MA_{\Bbb{R}}(\varphi_{j_1}+\dots+\varphi_{j_k}).
\end{equation}
The following result provides a key link between mixed volume and the (mixed) real Monge-Amp\'ere operator. We refer the reader to \cite[Proposition 3]{PasR} and \cite[Lemma 2.5]{BerBo} for the proof.
\begin{prop}\label{PR1}
Let $P_i\subset\R$ be a convex body and $\varphi_i$ be a convex function on $\R$ such that $\varphi_i-\varphi_{P_i}$ is bounded for each $i=1,\dots,m$. Then the total mass
$$\int_{\R}MA_{\Bbb{R}}(\varphi_1,\dots,\varphi_m)=MV_m(P_1,\dots,P_m).$$
\end{prop}
\subsection{Pluri-potential theory}\label{pp}
We let $\Bbb{C}^*:=\Bbb{C}\setminus\{0\}$ and $\|z\|$ denote the Euclidean norm of $z\in \C.$ For a convex body $P\subset\R,$ we denote
$$H_P(z):=\max_{J\in P}\log|z^{J}|$$ 
where we use the multi-dimensional notation $z^J:=z_1^{j_1}\dots z_m^{j_m}$ and $J=(j_1,\dots,j_m)\in\Z.$ Clearly, $H_P$ is a psh function on $\T.$  Indeed, $H_P$ coincides with $\varphi_P,$ the support function of $P$ in the logarithmic coordinates on $\T.$ Namely, letting 
$$Log:\T:\to\R$$
$$Log(z)=(\log|z_1|,\dots,\log|z_m|)$$
we see that $H_P(z)=\varphi_P\circ Log(z)$ for $z\in \T.$ For instance, if $P=\Sigma$ then $H_{\Sigma}(z)=\displaystyle\max_{i=1,\dots,m}\log^+|z_i|.$

We let $\mathcal{L}(\C)$ (respectively $\mathcal{L}_+(\C)$) denote the \textit{Lelong class} i.e. the set of psh functions $\psi$ on $\C$ such that $\psi(z)\leq \log^+\|z\|+C_{\psi}$ (respectively $\psi(z)-\log^+\|z\|$ is bounded). Following \cite{Ber1}, we also define the following classes of psh functions: 
$$\mathcal{L}_P:=\{\psi\in Psh(\T): \psi\leq H_P+C_{\psi}\ \text{on}\ \T\}$$
$$\mathcal{L}_{P,+}:=\{\psi\in \mathcal{L}_P: \psi\geq H_P+ C'_{\psi}\ \text{on}\ \T\}$$
We say that a function $\psi\in \mathcal{L}_P$ is \textit{$m$-circled} if  $\psi(z)=\psi(|z_1|,\dots,|z_m|),$ i.e. $\psi$ is invariant under the action of the real torus $(S^1)^m.$ We denote the set of all $m$-circled functions in $\mathcal{L}_P$ by $\mathcal{L}_P^c.$ The class $\mathcal{L}_P$ is a generalization of the Lelong class $\mathcal{L}(\C)$ which correspond to the case $P=\Sigma.$ Indeed, since every $\psi\in\mathcal{L}_{\Sigma}$ is locally bounded from above near points of the set $\{z\in\C: z_1\cdots z_m=0\},$ it extends to a psh function $\tilde{\psi}$ on $\C.$ Moreover, since $$\max_{J\in\Sigma}\log|z^J|=\max_{i=1,\dots,m}\log^+|z_i|\leq \log^+\|z\|$$ the extension $\tilde{\psi}\in \mathcal{L}(\C).$ 
 \\ \indent
The following lemma will be useful in the sequel.
\begin{lem}\label{helpy}
Let $P$ be a convex body and $\psi\in \mathcal{L}_{P,+}.$ Then for every $p\in P^{\circ}$ there exists $\kappa,C_{\psi}>0$ such that 
$$\psi(z)\geq \kappa\max_{j=1,\dots,m}\log|z_j|+\log|z^p|-C_{\psi} \ \ \text{for every}\ z\in\T.$$
\end{lem}
\begin{proof}
  Let $\varphi_P(x)$ denote the support function of $P.$ Fixing a small ball $B(p,\kappa)\subset P^{\circ},$ by definition we have $\varphi_P^{\star}\equiv0$ on $B(p,\kappa).$ Since $(\varphi_P^{\star})^{\star}=\varphi_P$ this implies that
\begin{eqnarray*} 
\varphi_P(x)\geq \sup_{q\in B(p,\kappa)}\langle q,x\rangle&=& \sup_{y\in B(0,1)}\langle \kappa y,x\rangle+\langle p,x\rangle\\
&=& \kappa\|x\|+\langle p,x\rangle
\end{eqnarray*}
 hence, using $H_P(z)=\varphi_P(Log(z))$ for $z\in\T$ we obtain
$$H_P(z)\geq \kappa\max_{j=1,\dots,m}|\log|z_j||+\log|z^p|$$
which implies the assertion. 
\end{proof}

\subsubsection{Global extremal function}\label{global}In this section, we let $K\subset\T$ be a non-pluripolar compact set and $q:\T\to \Bbb{R}$ be a continuous function. We define the \textit{weighted global extremal function} $V_{P,K,q}^*$ to be the usc regularization of 
$$V_{P,K,q}:=\sup\{\psi\in \mathcal{L}_P: \psi\leq q\ \text{on}\ K\}.$$ 
We remark that in the special case $P=\Sigma$ the function $V^*_{\Sigma,K,q}$ coincides with the weighted global extremal function defined in  \cite[Appendix B]{SaffTotik}. Moreover, specializing further, in the unweighted case (i.e. $q\equiv0$) $V^*_{\Sigma,K}$ is the pluricomplex Green function of $K$ (cf. \cite[\S5]{Klimek}). A standard argument shows that $V^*_{P,K,q}\in \mathcal{L}_{P,+}.$ In particular, $V_{P,K,q}^*\in Psh(\T)\cap L_{loc}^{\infty}(\T).$ The following example is a consequence of standard arguments (cf. \cite[\S 5]{Klimek}):
\begin{example}\label{torus}
For $P=[a,b]\subset \Bbb{R},\ K=S^1$ unit circle and $q\equiv0$ we have  $$V_{P,S^1}(z)=\max\{a\log|z|,b\log|z|\}=H_P(z)\ \text{for}\ z\in\Bbb{C}^*.$$ This implies that (more generally) for a convex polytope $P\subset\R$, $K=(S^1)^m\subset\T$ is the real torus and $q\equiv0$ the (unweighted) global extremal function $$V_{P,(S^1)^m}(z)=H_P(z)=\max_{J\in P}\log|z^J|\ \text{for}\ z\in\T.$$ In particular, $V_{P,(S^1)^m}$ is continuous.
\end{example}
\subsubsection{Complex Monge-Amp\'ere versus Real Monge-Amp\'ere} \label{CMA}
In what follows we denote $d=\partial+\bar{\partial}$ and $d^c=\frac{i}{2\pi}(\bar{\partial}-\partial)$ so that $dd^c=\frac{i}{\pi}\partial\bar{\partial}.$ It is well known that the relation between complex Monge-Amp\'ere of a m-circled psh function and the real Monge-Amp\'re of it (in the logarithmic coordinates) is given by 
\begin{equation}\label{eq1}
Log_*(MA_{\Bbb{C}}(\psi))=MA_{\Bbb{R}}(\varphi).
\end{equation}
That is for a Borel set  $E\subset \Bbb{R}^m$
$$\int_EMA_{\Bbb{R}}(\varphi)=\int_{Log^{-1}(E)}MA_{\Bbb{C}}(\psi).$$ 
Furthermore, by the results of \cite{BAR,BT2} the equality (\ref{eq1}) holds for every locally bounded $m$-circled psh function $\psi$ on $\T.$ Then (\ref{eq1}) together with polarization formula for complex Monge-Amp\'ere implies that
$$\bigwedge_{i=1}^mdd^c\psi_i=\frac{1}{m!}\sum_{j=1}^m\sum_{1\leq i_1\leq \dots\leq i_j}(-1)^{m-j}MA_{\Bbb{C}}(\psi_{i_1}+\dots+\psi_{i_j})$$
and (\ref{rpol}) implies that for locally bounded m-circled psh functions $\psi_1,\dots,\psi_m$
$$Log_*(\bigwedge_{i=1}^mdd^c\psi_i)=MA_{\Bbb{R}}(\varphi_{1},\dots,\varphi_{m})$$
where $\varphi_i(x)=\psi_i(z)$ is the corresponding convex function defined as above. Thus, the following is an immediate consequence of Proposition \ref{PR1}:
\begin{prop}\label{mass}
Let $\psi_i\in \mathcal{L}^c_{P_i,+}$ for $i=1,\dots,m$ then the total mass of the mixed complex Monge-Amp\'ere
$$\int_{\T}\bigwedge_{i=1}^mdd^c\psi_i=MV_m(P_1,\dots,P_m).$$ 
\end{prop}
By Example \ref{torus} and Proposition \ref{mass} we obtain:
\begin{example}\label{torus2}
Let $P_i\subset\R$ be convex polytopes for $i=1,\dots,m,$ $K=(S^1)^m$ is the real torus and $q\equiv0.$ Then the mixed complex Monge-Amp\'ere
$$\bigwedge_{i=1}^mdd^c(V_{P_i,K})=\frac{MV_m(P_1,\dots,P_m)}{(2\pi)^m}d\theta_1\dots d\theta_m.$$ 
\end{example}

Recall that the extremal function $V:=V^*_{P,K,q}$ is a locally bounded psh function on $\T$. Thus, by \cite{BT2} its complex Monge-Amp\'ere measure
$$MA_{\Bbb{C}}(V):=dd^c(V)\wedge\dots\wedge dd^c(V)$$ is well defined and does not charge pluripolar subsets of $\T$.
We denote the support of complex Monge-Amp\'ere of the extremal function by $supp(MA_{\Bbb{C}}(V)).$ 
The following result is classical and follows from \cite[Proposition 3]{PasR} and \cite[Lemma 2.5]{BerBo}. 

\begin{prop}\label{support}
Let $P$ be a convex body and $(K,q)$ be a regular weighted compact set. Then $$supp(MA_{\Bbb{C}}(V_{P,K,q}))\subset\{z\in K:V_{P,K,q}(z)=q(z)\}.$$
In particular, if $K$ is circled and $q\in \mathcal{L}_{P,+}^c\cap\mathcal{C}^2(\T)$ then
\begin{equation}
Log(supp(MA_{\Bbb{C}}(V)))\subset \nabla\varphi^{-1}(P^{\circ})
\end{equation} where $\varphi$ is the convex function defined by relation $q(z)=\varphi(Log(z)).$
\end{prop}
A remarkable property of the Lelong class functions $\psi\in \mathcal{L}(\C)\cap L^{\infty}_{loc}(\C)$ is that the total mass $\int_{\C}MA_{\Bbb{C}}(\psi)\leq 1.$ Moreover, if $\psi\in \mathcal{L}_+(\C)$
\begin{equation}\label{cmamv}
\int_{\C}MA_{\Bbb{C}}(\psi)= \int_{\Bbb{C}^m}MA_{\Bbb{C}}(\frac12\log(1+\|z\|^2)))=1
\end{equation}
which was observed in \cite{Taylor}. The equality (\ref{cmamv}) is a consequence of comparison theorem (see \cite[\S 5]{Klimek} for the details and references). 

 In what follows we let $\omega:=\frac12dd^c\log(1+\|z\|^2)$ denote the restriction of the Fubini-Study form to $\T$ and
$$\varpi:=dd^cH_{\Sigma}(z)=\displaystyle dd^c(\max_{i=1,\dots,m}\log^+|z_i|).$$
 We also denote the product of annulli by
$$A_{\rho,R}:=\{z\in\T:  \rho<|z_i|<R\ \text{for each}\ i=1,\dots,m\}\ \text{for}\ \rho,R>0.$$ 
Next, we obtain a generalized version of \cite{Taylor} to our setting:  
\begin{prop}\label{mixma}
Let $P_i\subset \R$ be a convex body and $u_i,v_i\in \mathcal{L}_{P_i}\cap L^{\infty}_{loc}(\T)$ such that
$$u_i(z)\leq v_i(z)+C_i \ \text{for}\ z\in\T$$ for each $i=1,\dots,k.$ Then the total masses
$$\int_{\T}\bigwedge_{i=1}^kdd^cu_i\wedge\varpi^{m-k}\leq \int_{\T}\bigwedge_{i=1}^kdd^cv_i\wedge\varpi^{m-k}.$$
In particular, if $u_i\in\mathcal{L}_{P_i,+}$ each $i=1,\dots,k$ then the total mass of the mixed Monge-Amp\'ere
$$\int_{\T}dd^cu_1\wedge\dots\wedge dd^cu_k\wedge\varpi^{m-k}=MV_m(P_1,\dots,P_k,\Sigma,\dots,\Sigma).$$
\end{prop} 
  
\begin{proof}
Since the complex Monge-Amp\'ere is a symmetric operator by replacing $v_i$ with $u_i$ successively in the $i^{th}$ step, it is enough to prove the assertion for the case $v_i=u_i$ for $2\leq i\leq k$. \\ \indent
We fix a convex body $Q\subset\R$ such that $0\in Q^{\circ}$. Then by Lemma \ref{helpy} and replacing $v_1$ by $v_1':=v_1+\epsilon H_Q$ for $\epsilon>0$ if necessary, we may assume that 
 $$u_1-v'_1\to -\infty$$ as $\|z\|\to \infty$ as well as  $|z_j|\to0$ for some $j\in\{1,\dots,m\}.$  Now, we define 
 $$\psi_N=\max\{u_1,v'_1-N\} .$$
 Note that $\psi_N=v'_1-N$ near the boundary of the set $A_{\rho,R}$ for sufficiently large $R>0$ and small $\rho>0.$ Thus,  by Stokes' theorem we obtain
 \begin{eqnarray*}
 \int_{\T}dd^cv'_1\wedge \bigwedge_{i=2}^kdd^cv_i\wedge \varpi^{m-k} &\geq &  \int_{A_{\rho,R}}dd^cv'_1\wedge \bigwedge_{i=2}^kdd^cv_i\wedge\varpi^{m-k} \\
 &=& \int_{A_{\rho,R}} dd^c\psi_N\wedge \bigwedge_{i=2}^kdd^cv_i\wedge \varpi^{m-k}.
 \end{eqnarray*}
Since $\psi_N$ decreases to $u_1$ as $N\to\infty,$ by Bedford-Taylor theorem \cite{BT2} on continuity of Monge-Amp\'ere measures along decreasing  sequences we infer that
 $$ \int_{\T}dd^cv'_1\wedge \bigwedge_{i=2}^kdd^cv_i\wedge\varpi^{m-k}\geq \int_{A_{\rho,R}}dd^cu_1\wedge \bigwedge_{i=2}^kdd^cv_i\wedge\varpi^{m-k}. $$ 
Finally, since $R\gg1,\rho>0 $ and $\epsilon>0$ are arbitrary letting $R\to \infty,\rho\to0$ and $\epsilon \to 0$ in $v_1'=v_1+\epsilon H_Q$ respectively, we obtain the first assertion.\\ \indent
  To prove the second assertion we let $v_i=H_{P_i}$ and apply the first part together with Proposition \ref{mass}. 

\end{proof}
\begin{rem}\label{rem}
We remark that the condition $u_i\in L^{\infty}_{loc}(\T)$ in Proposition \ref{mixma} is used to make sure that the mixed complex Monge-Amp\'ere is well defined. Thus, we infer that for $\psi\in \mathcal{L}_P$ the total mass of $MA_{\Bbb{C}}(\psi)$ is finite as soon as it is well defined on $\T.$ Note that by Bertini's theorem for generic $f_N^i\in Poly(NP_i)$ their zero sets $Z_{f_N^i}$ are smooth and intersect transversely.  It follows from \cite[\S III, Theorem 4.5]{DemBook} that for systems $(f^1_N,\dots,f^k_N)$ in general position the current of integration
$$[Z_{f_N^1,\dots,f_N^k}]=dd^c\log|f_N^1|\wedge\dots \wedge dd^c\log|f_N^k| $$ is well defined and has locally finite mass. Thus, it follows from Proposition \ref{mixma} that
\begin{eqnarray}\label{expo}
\frac{1}{N^k}\int_{\T} [Z_{f^1_N,\dots,f_N^k}]\wedge \omega^{m-k} &\leq&MV_m(P_1,\dots,P_k,\Sigma,\dots,\Sigma) 
\end{eqnarray}
which was also observed in \cite[Cor. 6.1]{Rashkovskii2} when $P\subset \R_{\geq 0}$.
\end{rem}
\subsection{A Siciak-Zaharyuta theorem}
We start with a basic result which is an easy consequence of Cauchy's estimates on the product of annulli 
$$A_{\rho,R}:=\{z\in\T:  \rho<|z_i|<R\ \text{for each}\ i=1,\dots,m\}\ \text{for}\ 0<\rho<R$$ together with a Liouville type argument.
\begin{prop}\label{poly}
Let $P\subset\R$ be an integral polytope and $f\in \mathcal{O}(\T)$ such that
$$\int_{\T}|f(z)|^2e^{-2NH_P(z)}(1+|z|^2)^{-r}dz<\infty$$ for some $0\leq r\ll 1.$ Then $f$ is a Laurent polynomial such that its support $S_f\subset NP.$ 
\end{prop}

Throughout this section we denote $V:=V^*_{P,K,q}$ where $K$ and $q$ as in (\ref{global}) and $P\subset \R$ is an integral polytope with non-empty interior. Next, we define
$$\Phi_N:=\sup_{z\in\T}\{|f_N(z)|: f_N\in Poly(NP)\ \text{and}\ \max_{z\in K}|f_N(z)|e^{-Nq(z)}\leq1\}$$
Note that $\Phi_N.\Phi_M\leq \Phi_{N+M}$ which implies that $\lim_{N\to \infty}\frac1N\log\Phi_N(z)$ exists for $z\in \T$. Observe also that for each $f_N\in Poly(NP)$ the function $\frac{1}{N}\log|f_N(z)|$ belongs to $\mathcal{L}_P.$ Hence, $\lim_{N\to \infty}\frac1N\log\Phi_N\leq V$ on $\T.$  If $P$ is the unit simplex $\Sigma$ then it follows from seminal works of Siciak and Zaharyuta (see \cite{Klimek} for details) that $$\lim_{N\to \infty}\frac1N\log\Phi_N= V_{K,q}$$ point-wise on $\C$. We obtain a slightly stronger version of this result in the present setting:
\begin{thm}\label{extremal}
Let $P\subset\R$ be an integral polytope with non-empty interior and $(K,q)$ be a regular weighted compact set. Then
$$V_{P,K,q}=\lim_{N\to\infty}\frac{1}{N}\log\Phi_N$$ locally uniformly on $\T.$
\end{thm}

\begin{proof}
For a given compact set $X\subset \T$ we need to show that for every $\epsilon>0$ there exists $N_0\in\Bbb{N}$ such that
$$0\leq V(z)-\frac{1}{N}\log\Phi_N(z)<\epsilon$$ for every $z\in X$ and $N\geq N_0.$ To this end, fix  $z_0\in X,$ and $B(z_0,\delta)\subset\T$ be a small ball centered at $z_0$ such that for every $z\in B(z_0,\delta)$
\begin{equation}\label{end}
|V(z)-V(z_0)|<\frac{\epsilon}{2}.
\end{equation}
First, we assume that $V$ is a smooth function on $\T.$ We also let $\chi$ be a test function with compact support in $B(z_0,\delta)$ such that $\chi\equiv1$ on $B(z_0,\frac{\delta}{2}).$ For a fixed point $p\in P^{\circ},$ we define  
$$\psi_N(z):=(N-\frac{m}{\kappa})(V(z)-\frac{\epsilon}{2})+\frac{m}{\kappa}\log|z^p|+m\max_{j=1,\dots,m}\log|z_j-z_{0,j}|$$ where $\kappa>0$ is as in Lemma \ref{helpy} and $\frac{m}{\kappa}\ll N.$ Note that $\psi_N$ is psh on $\T,$ smooth away $z_0$ and its Lelong number $\nu(\psi_N,z_0)=m.$ 
Since $\T$ is pseudoconvex by H\"ormander's $L^2$-estimates \cite[\S VIII]{DemBook} for every $r\in(0,1]$ there exists a smooth function $u_N$ on $\T$ such that $\bar{\partial}u_N=\bar{\partial}\chi$ and  
$$\int_{\T}|u_N|^2e^{-2\psi_N}(1+|z|^2)^{-r}dz\leq \frac{4}{r^2}\int_{\T}|\bar{\partial}\chi|^2e^{-2\psi_N}(1+|z|^2)dz.$$
Note that the $(0,1)$ form $\bar{\partial}\chi$ is supported in $B(z_0,\delta)\backslash B(z_0,\frac{\delta}{2})$ therefore both integrals are finite. Since $\nu(\psi_N,z_0)= m$ this implies that $u_N(z_0)=0.$ Moreover, by Lemma \ref{helpy} we obtain
\begin{equation}
\int_{\T}|u_N|^2e^{-2N(V-\frac{\epsilon}{2})}(1+|z|^2)^{-r}dz\leq C_1e^{-2N(V(z_0)-\epsilon)}
\end{equation}
 where $C_1>0$ does not depend on $N.$
Next, we let $f_N:=\chi-u_N.$ Then $f_N$ is a holomorphic function on $\T$ such that $f(z_0)=1.$ Furthermore,
\begin{equation}\label{holo}
\int_{\T}|f_N|^2e^{-2N(V-\frac{\epsilon}{2})}(1+|z|^2)^{-r}dz\leq C_2e^{-2N(V(z_0)-\epsilon)}
\end{equation}
where $C_2>0$ is independent of $N.$ Then using $V\in\mathcal{L}_{P,+}$ again we see that 
$$\int_{\T}|f_N|^2e^{-2NH_P}(1+|z|^2)^{-r}dz<\infty $$ and taking $r>0$ small, Proposition \ref{poly} implies that $f_N$ is a polynomial such that $S_{f_N}\subset NP.$\\ \indent
Finally, if $V$ is not smooth on $\T$ then we approximate $V$ by smooth psh functions $V_{t}:=\varrho_{t}\star V\geq V$ (where $\varrho_t$ is an approximate identity) on an increasing sequence of pseudoconvex domains $\Omega_{t}\Subset \T$ as $t\to0.$ Since $V$ is continuous, $V_{t}$ converges to $V$ locally uniformly. Thus, we obtain functions $f_{N,t}\in\mathcal{O}(\Omega_{t})$ satisfying (\ref{holo}) and extract a convergent subsequence $f_{N,t_k}\to f_N$ as $t_k\to 0$ where $f_N$ is a holomorphic function which satisfies (\ref{holo}) and hence $f_N\in Poly(NP)$.\\ \indent
 Next, we show that a multiple of $f_N$ satisfies the necessary growth condition on $K$. Since $V-q$ is continuous, by compactness of $K\subset \T$ there exists $\rho>0$ such that $K_{\rho}:=\{z\in \T: dist(z,K)<\rho\}\subset \T$ and for every $z\in K$  
 $$|q(y)-q(z)|<\frac{\epsilon}{2}$$
 and
 $$q(y) > V(y)-\frac{\epsilon}{2}$$
 for every $y\in B(z,\rho)\subset \T .$
 Now, let $$C_r:=\min_{z\in K_{\rho}}(1+|z|^2)^{-r}$$
 then applying sub-mean value inequality to subharmonic function $|f_N(z)|^2$ on $B(z,\rho)$  we obtain
 \begin{eqnarray*}
 C_r|f_N(z)|^2e^{-2Nq(z)} &\leq& C_3\int_{B(z,\rho)}|f_N(y)|^2e^{-2N(V(y)-\frac{\epsilon}{2})+(q(z)-q(y))}(1+|y|^2)^{-r}dy\\
 &\leq&C_3\int_{\T}|f_N(y)|^2e^{-2N(V(y)-\epsilon)}(1+|y|^2)^{-r}dy\\
 &\leq& C_4e^{-2N(V(z_0)-\epsilon)} 
 \end{eqnarray*}
where $C_4>0$ is as above independent of $N.$ Thus, replacing $f_N$ by $F_N:=\sqrt{\frac{C_r}{C_4}}e^{N(V(z_0)-\epsilon)}f_N$ we see that 
 $$\max_{z\in K}|F_N(z)e^{-q(z)}|\leq 1.$$ and 
 \begin{equation}\label{end2}
 \frac{1}{N}\log|F_N(z_0)|=V(z_0)-\epsilon+\frac{1}{2N}\log (\frac{C_r}{C_4}).
 \end{equation}
It remains to show uniform convergence on the compact set $X$. To this end, we utilize some ideas from \cite{BloomS}. Choosing $N_0$ large enough such that for $N\geq N_0$
$$ \frac{1}{2N}\log (\frac{C_r}{C_4})<\epsilon,\ \ \ \frac{1}{N}V(z_0)<\epsilon\ \text{and}\
-\epsilon<\frac1N\log\Phi_1(z)<\epsilon$$ for $z\in B(z_0,\delta)$ where  $\delta,\epsilon> 0$ as in (\ref{end}). 
Moreover, by shrinking $B(z_0,\delta)$ if necessary we may assume that 
$$\frac{1}{N_0}\log\Phi_{N_0}(z_0)-\frac{1}{N_0}\log\Phi_{N_0}(z)<\frac{\epsilon}{2}$$ 
on $B(z_0,\delta).$ Then for $N\geq N^2_0$ writing $N=kN_0+j$ with $0\leq j< N_0$ we obtain 
\begin{eqnarray*}
\frac1N\log\Phi_N &\geq& \frac{1}{kN_0+j}\log\Phi_{kN_0}+\frac{1}{kN_0+j}\log \Phi_j\\
&\geq& \frac{k}{kN_0+j}\log\Phi_{N_0}+\frac{j}{kN_0+j}\log\Phi_1\\
&\geq& \frac{1}{N_0+\frac{j}{k}}\log\Phi_{N_0}-\epsilon 
\end{eqnarray*}
on $B(z_0,\delta).$ Then 
\begin{eqnarray*}
V(z)-\frac{1}{N}\log\Phi_N &\leq& V(z)- \frac{1}{N_0+\frac{j}{k}}\log\Phi_{N_0}+\epsilon \\
&\leq& V(z)- \frac{1}{N_0}\log\Phi_{N_0}+ \frac{j}{N_0(kN_0+j)}V(z)+\epsilon \\
&\leq& V(z)-\frac{1}{N_0}\log \Phi_{N_0}+2\epsilon .
\end{eqnarray*}
Now, by (\ref{end2}) for $N\geq N_0$ 
\begin{eqnarray*}
0&\leq& V(z)-\frac{1}{N}\log\Phi_{N}(z) \\
&\leq& (V(z_0)-\frac{1}{N_0}\log\Phi_{N_0}(z_0))+(V(z)-V(z_0)) + (\frac{1}{N_0}\log\Phi_{N_0}(z_0)-\frac{1}{N_0}\log\Phi_{N_0}(z))+2\epsilon\\
&\leq& 2\epsilon+\frac{\epsilon}{2}+\frac{\epsilon}{2}+2\epsilon=5\epsilon
\end{eqnarray*}
for $z\in B(z_0,\delta)$ and $N\geq N^2_0.$ Finally, covering the compact set $X$ with finitely many $B(z_i,\delta_i)$ we see that 
$$0\leq V-\frac1N\log\Phi_N \leq 5\epsilon\ \text{for every}\ N\geq\max_iN^2_i.$$
on $X.$ This finishes the proof.
\end{proof}

\subsection{Bernstein-Markov measures}\label{BMp}
 Next, we turn our attention to $L^2$ space of weighted polynomials. A measure $\tau$ supported in $K$ is called a \textit{Bernstein-Markov measure} for the triple $(P,K,q)$ if it satisfies the weighted Bernstein-Markov inequality: there exists constants $M_N>0$ such that for every $f_N\in Poly(NP)$
$$\max_K|f_Ne^{-Nq}|\leq M_N\|f_Ne^{-Nq}\|_{L^2(\tau)}$$
and $\limsup_{N\to \infty}(M_N)^{\frac1N}=1.$ This roughly means that $\sup$-norm and $L^2(\tau)$-norm on $Poly(NP)$ are asymptoticly equivalent. 
We remark that if $P\subset p\Sigma$ then any BM measure (for polynomials of degree at most N) induces a BM measure for our setting. For instance, for $P=\Sigma$ it follows from  \cite{NgZ} that the complex Monge-Amp\'ere of the unweighted (i.e. $q\equiv0$) global extremal function $V^*_{K}$  of a regular compact set $K$ satisfies BM inequality.\\ \indent 
 Next, we fix an orthonormal basis $\{F_j\}_{j=1}^{d_N}$ for $Poly(NP)$ with respect the inner product induced from $L^2(e^{-2Nq}\tau).$ Then associated Bergman kernel is given by
$$S_N(z,w)=\sum_{j=1}^{d_N}F_j(z)\overline{F_j(w)}$$ where $d_N=\dim Poly(NP)$.

The following result was proved in \cite[Lemma 3.4]{BloomS} for the case $P=\Sigma.$ Their argument generalizes to our setting mutatis-mutandis.

\begin{prop}\label{BS}
Let $P$ be an integral polytope with non-empty interior, $K\subset \T$ be a compact set and $q:\T\to\Bbb{R}$ be continuous weight function such that $V:=V_{P,K,q}$ is continuous. If $\tau$ be a BM measure supported on $K$ then $$\frac{1}{2N}\log S_N(z,z)\to V_{P,K,q}(z)$$ uniformly on compact subsets of $\T$
\end{prop}

\section{Expected distribution of zeros}
Recall that if $P\subset\R$ is an integral polytope then 
\begin{equation}\label{Eh}
\#(NP\cap\Z)=\dim(Poly(NP))=Vol(P)N^m+o(N^m)
\end{equation} where the latter is known as Ehrhart polynomial of $P$ \cite{Eh}.   \\ \indent  We identify each $f_N\in Poly(NP)$ with a point in $\Bbb{C}^{d_N}$ by 
$$\Psi_N:Poly(NP)\to \Bbb{C}^{d_N}$$
$$f_N=\sum_{j=1}^{d_N}a_j^NF_j\to a^N:=(a_j^N).$$
First, we prove that conditions $(A1),(A2)$ and $(A3)$ hold for random sparse polynomials with iid coefficients under a mild moment condition:
\begin{prop}[iid coefficients]\label{iid} Assume that $a_j$ are iid complex (or real) valued random variables whose distribution law $\textbf{P}$ is of the form $\textbf{P}=\phi(z)dz$ (or $\textbf{P}=\phi(x)dx$) where $\phi$ is a real valued bounded function satisfying 
$$\textbf{P}\{z\in\C: \log|z|>R\}\leq \frac{C}{R^{\rho}}\ \text{for}\ R\geq1$$ for some $\rho>m+1.$ Then the $d_N$-fold product measure $\sigma_N$ on $\Bbb{C}^{d_N}$ induced by $\textbf{P}$ satisfies conditions $(A1),(A2)$ and $(A3).$
\end{prop}
\begin{proof}
$(A1)$ is a direct consequence of \cite[Lemma 3.1]{B6}. In order to show $(A2)$ holds, we note that
 for $N\gg 1$ and $\epsilon>0$ 
\begin{eqnarray*}
\sigma_N\{a\in \Bbb{C}^{d_N}: \|a\|>e^{\epsilon N}\} &\leq& \sigma_N\{a\in \Bbb{C}^{d_N}:\|a\|>\sqrt{d_N}e^{\frac{\epsilon N}{2}}\}\\
&\leq & \sigma_N\{a\in \Bbb{C}^{d_N}:|a_j|>e^{\frac{\epsilon}{2}N}\ \text{for some}\ j\}\\
&\leq& \frac{C_{\epsilon}d_N}{N^{\rho}}
\end{eqnarray*}
where the latter is summable. 

\par Finally, for $u\in S^{2d_N-1}$ fixed we may assume that $|u_1|\geq \frac{1}{\sqrt{d_N}}$ and applying the change of variables $w_1=\sum_{j=1}^{d_N}a_ju_j, w_2=a_2,\dots, w_{d_N}=a_{d_N}$ we see that
\begin{eqnarray*}
\sigma_N\{a\in \Bbb{C}^{d_N}: |\la a,u\ra|<e^{-tN}\} &=&\int_{\Bbb{C}^{d_N-1}}\int_{|w_1|\leq e^{-tN}}\frac{1}{|u_1|^2}\phi(\frac{w_1-w_2 u_2-\dots -w_{d_N} u_{d_N}}{u_1})d\lambda(w_1) d\sigma_{N-1}\\
&\leq & C\pi d_Ne^{-2tN}.
\end{eqnarray*}
Since the latter is summable $(A3)$ follows.
\end{proof}
Let $X$ be a complex manifold and $\sigma$ be a positive measure on $X$. Following \cite{DNS}, we say that $\sigma$ is (locally) moderate if for any open set $U\subset X,$ compact set $K\subset U$ and a compact family $\mathscr{F}$ of psh functions there exists constants $c,\alpha>0$ such that
\begin{equation}\label{mode1}
\int_Ke^{-\alpha\psi}d\sigma\leq c\ \ \forall \psi\in\mathscr{F}.
\end{equation}  
The existence of $c,\alpha$ in (\ref{mode1}) is equivalent to existence of $c'\alpha'>0$ satisfying
\begin{equation}\label{mode2}
\sigma\{z\in K: \psi(z)<-t\}\leq c'e^{-\alpha't}
\end{equation}
for $t\geq0$ and $\psi\in\mathscr{F}.$
Next, we observe that moderate measures also fall into frame work of our main results.
\begin{prop}[Moderate measures]\label{moderate}
Let $\sigma_N$ be a moderate measure supported on $S^{2d_N-1}$ then $\sigma_N$ satisfies conditions $(A1),(A2)$ and $(A3)$.
\end{prop}
\begin{proof}
Since $supp(\sigma_N)\subset S^{2d_N-1}$ condition $(A2)$ is automatically satisfied. Moreover, for every $u\in S^{2d_N-1}$ the function $\psi_u:\Bbb{C}^{d_N}\to \R$
$$ \psi_u(w)=\log|\langle w,u\rangle|$$ is psh and $\sup_{S^{2d_N-1}}\psi_u=0.$ Since $\sigma_N$ is moderate, letting $\mathscr{F}=\{\psi_u:u\in S^{2d_N-1}\}$ it follows from (\ref{mode2}) that there exists $C,\alpha>0$ such that
$$\sigma_N\{w\in \Bbb{C}^{d_N}:\log |\langle w,u\rangle|<-R\}\leq Ce^{-\alpha R}\ \text{for}\ R>0$$ for every $u\in S^{2d_N-1}.$ This verifies $(A3).$ 

Since,
$$\int_{\Bbb{C}^{d_N}}|\log|\la a,u\ra||d\sigma_N(a)\leq 1+\int_0^{\infty} \sigma_N\{a\in \Bbb{C}^{d_N}: |\la a,u\ra|<e^{-tN}\} dt $$
$(A1)$ follows.
\end{proof}

For a complex manifold $Y$ we denote the set of bidegree $(m-k,m-k)$ test forms i.e. smooth forms with compact support by $\mathcal{D}_{m-k,m-k}(Y).$ Then a bidegree $(k,k)$ current is a continuous linear functional on $\mathcal{D}_{m-k,m-k}(Y)$ with respect to the weak topology. We denote the set of bidegree $(k,k)$ currents by $\mathcal{D}^{k,k}(Y).$ We refer the reader to the manuscript \cite{DemBook} for detailed information regarding the theory of currents.\\ \indent
For each $f_N\in Poly(NP)$ we let $[Z_{f_N}]$ denote the current of integration along regular points of the zero locus of $f_N$ and denote the action of it on a test form $\Theta\in \mathcal{D}_{m-1,m-1}(Y)$ by $\langle [Z_{f_N}],\Theta\rangle.$ Then the \textit{expected zero current} of random Laurent polynomials $f_N\in Poly(NP)$ was defined in the introduction by
$$\langle\Bbb{E}[Z_{f_N}],\Theta\rangle=\int_{Poly(NP)}\langle[Z_{f_N}],\Theta\rangle d\sigma_N(f_N).$$

Next lemma provides a link between Bergman kernels and expected distribution of zeros of random sparse polynomials. 
\begin{prop}\label{helpex}
Let $P\subset\R$ be an integral polytope with non-empty interior then there exists a real closed $(1,1)$ current $T_N \in\mathcal{D}^{(1,1)}(\T)$ such that
for every test form $\Theta\in\mathcal{D}_{(m-1,m-1)}(\T)$ 
$$\frac1N\langle\Bbb{E}[Z_{f_N}],\Theta\rangle=\frac{1}{2N}\langle dd^c(\log S_N(z,z)),\Theta\rangle+\langle T_N,\Theta\rangle$$
and $T_N\to 0$ weakly as $N\to \infty.$ In particular,
$$\frac1N\Bbb{E}[Z_{f_N}] \to dd^cV_{P,K,q}$$ weakly as $N\to \infty.$
\end{prop} 
\begin{proof}
It follows from Poincar\'e-Lelong formula that
$$[Z_{f_N}]=dd^c\log|f_N|.$$
 Writing $f_N=\sum_{j=1}^{d_N}a_jF^N_j=:\langle (a^N),(F^N_j)\rangle$ where $\{F^N_j\}$ is the fixed ONB for $Poly(NP)$ and letting $u_N(z):=\big(\frac{F^N_1(z)}{\sqrt{S_N(z,z)}},\dots, \frac{F^N_{d_N}(z)}{\sqrt{S_N(z,z)}}\big)$ for $z\in\T,$ by Fubini's theorem we obtain
\begin{eqnarray*}
\frac1N\langle\Bbb{E}[Z_{f_N}],\Theta\rangle&=& \int_{\Bbb{C}^{d_N}}\langle \frac{1}{2N}dd^c\log S_N(z,z),\Theta\rangle d\probn(a^N)+\frac1N\int_{\T}dd^c\Theta\int_{\Bbb{C}^{d_N}}\log|\langle a^N,u_N(z)\rangle| d\probn(a^N) \\
&=:& \frac{1}{2N}\langle dd^c(\log S_N(z,z)),\Theta\rangle+\langle T_N,\Theta\rangle
\end{eqnarray*}
Moreover,
$$|\langle T_N,\Theta\rangle| \leq \frac1N\|dd^c\Theta\|_{\infty} \sup_{u\in S^{2d_N-1}}|\int_{\Bbb{C}^{d_N}}\log|\la a,u\ra|d\sigma_N(a)|$$ where $\|dd^c\Theta\|_{\infty}$ denotes the sum of the sup norms of the coefficients of the smooth form $dd^c\Theta$. Thus, the first assertion follows from (A1).
\par Now, the second assertion is an immediate consequence of Proposition \ref{BS}
\end{proof}
For an algebraic submanifold $Y\subset \T$ we let $Z_{f_{|Y}}:=\{z\in Y:f(z)=0\}$ denote the restriction of the zero locus of $f$ on $Y.$ The following is a well known probabilistic version of Poincar\'e-Lelong formula (see \cite[\S 5]{SZ1} and \cite[\S 3]{B6}): 
\begin{prop}\label{propex}
The expected zero current of independent random Laurent polynomials $f^i_N\in Poly(NP_i)$, $1\leq i\leq k$ is given by
$$\Bbb{E}[Z_{f_N^1,\dots,f_N^k}]= \bigwedge_{i=1}^k\Bbb{E} [Z_{f_N^i}]$$
\end{prop}

Now, we are ready to prove Theorem \ref{main}.
\begin{proof}[Proof of Theorem \ref{main}]
 Note that for every continuous $(m-1,m-1)$ form $\Theta$ with compact support in $\T$
$$ |\la Z_{f_N},\Theta\ra|\leq \la Z_{f_N},\omega^{m-1}\ra \|\Theta\|_{\infty}\leq MV_m(P_1,\Sigma,\dots,\Sigma)\|\Theta\|_{\infty}$$ by approximating $\Theta$ with smooth forms it is enough to consider test forms on $\T.$  

\par We prove the theorem by induction on bidegrees. The case $k=1$ was obtained in Proposition \ref{helpex}.

\par Let us denote
 \begin{equation}\label{an}
 \alpha^j_N:=\frac{1}{2N}dd^c\log S^j_N(z,z)
 \end{equation} 
where $S_N^j(z,w)$ is the Bergman kernel for $Poly(NP_j).$ We claim that for every test form $\Theta\in\mathcal{D}_{m-k,m-k}(\T)$
$$\frac{1}{N^k}\la \Bbb{E}[Z_{f^1,\dots,f^k}],\Theta\ra=\la \bigwedge_{j=1}^k\alpha^j_N,\Theta\ra+ \langle T_N^k,\Theta\rangle$$ where $T_N^k$ is a real closed $(k,k)$ current such that $T^k_N\to0$ weakly as $N\to \infty.$
Assume that the the claim holds for $k-1.$ By Bertini's theorem for generic $f^k_N\in Poly(NP_k)$ its zero locus $Z_{f^k_N}$ is smooth and has codimension one in $\T.$ Then, using the notation in Proposition \ref{helpex} and by applying induction hypothesis
\begin{eqnarray*}
 \frac{1}{N^k}\int_{Z_{f_N^k}}\la [Z_{f_N^1,\dots,f_N^{k-1}}],\Theta\ra d\sigma_N(f_N^1)\dots d\sigma_N(f_N^{k-1}) &=&\frac{1}{N^k} \int_{Z_{f^k_N}}\la \Bbb{E}[Z_{f_N^1,\dots,f_N^{k-1}}],\Theta\ra  \\
 &=& \int_{Z_{f^k_N}} (\bigwedge_{j=1}^{k-1}\alpha_N^{j} \wedge \Theta + \la T^{k-1}_N,\Theta\ra) 
\end{eqnarray*}
where 
\begin{eqnarray*}
|\la (T^{k-1}_N)_{|_{Z_{f^k_N}}},\Theta_{|_{Z_{f^k_N}}}\ra| &\leq&  \|T^{k-1}_N\|\|dd^c \Theta_{|_{Z_{f^k_N}}}\|_{\infty}\\ 
&\leq& \|T^{k-1}_N\|\|dd^c\Theta\|_{\infty}\int_{Z_{f^k_N}}\omega^{m-1}\\
&\leq&  \|T^{k-1}_N\|\|dd^c\Theta\|_{\infty}MV_m(P_k,\Sigma,\dots,\Sigma)
\end{eqnarray*}
and $\|T^{k-1}_N\|$ denotes the mass of $T^{k-1}_N.$ Now, taking the average over $f^k_N\in Poly(NP_k)$ and using the estimate for the case $k=1$ we obtain
\begin{eqnarray*}
\frac{1}{N^k}\la \Bbb{E}[Z_{f_N^1,\dots,f_N^k}],\Theta\ra 
&=& \la \bigwedge_{j=1}^k \alpha^j_N,\Theta\ra+ \la T_N^1, \bigwedge_{j=1}^{k-1}\alpha_N^{j} \wedge \Theta\ra +\int_{Poly(NP_k)}\la (T^{k-1}_N)_{|_{Z_{f^k_N}}},\Theta_{|_{Z_{f^k_N}}}\ra d\sigma_N(f^k_N)\\
&=& \la \bigwedge_{j=1}^k\alpha^j_N,\Theta\ra+ C_{\Theta,N}
\end{eqnarray*}
where  $$ C_{\Theta,N} = \la T_N^1, \bigwedge_{j=1}^{k-1}\alpha_N^{j} \wedge \Theta\ra +\int_{Poly(NP_k)}\la (T^{k-1}_N)_{|_{Z_{f^k_N}}},\Theta_{|_{Z_{f^k_N}}}\ra d\sigma_N(f^k_N)$$ then by Proposition \ref{helpex} we have 
\begin{eqnarray*}
|C_{\Theta,N}| &\leq& |\la T_N^1, \bigwedge_{j=1}^{k-1}\alpha_N^{j} \wedge \Theta\ra| +|\int_{Poly(NP_k)}\la T^{k-1}_N,\Theta_{|_{Z_{f^k_N}}}\ra d\sigma_N(f^k_N)|\\
&\leq& \|T^1_N\|\|dd^c\Theta\|_{\infty} MV_m(P_1,\dots,P_{k-1},\Sigma,\dots,\Sigma) +\|T^{k-1}_N\|\|dd^c{\Theta}\|_{\infty}MV_m(P_k,\Sigma,\dots,\Sigma).
\end{eqnarray*}
 Thus, the assertion follows from the above estimate, induction hypothesis and the uniform convergence of Bergman kernels to weighted global extremal function (Proposition \ref{BS}) together with a theorem of Bedford and Taylor \cite{BT2} on convergence of Mong\'e-Ampere measures along uniformly convergent sequences of psh functions.
 \end{proof}

\section{Self-averaging}\label{ae}
 In this section we prove Theorem \ref{sa}. Let $\Bbb{P}^m$ denote the complex projective space and $\omega_{FS}$ is the Fubini-Study form. We also let $dV$ denote the volume form induced by $\omega_{FS}.$ Recall that an usc function $\varphi\in L^1(\Bbb{P}^m,dV)$ is called $\omega_{FS}$-psh if $\omega_{FS}+dd^c\varphi\geq0$ in the sense of currents. It is well know that (see eg \cite{DemBook}) there is a 1-1 correspondence between Lelong class psh function $\mathcal{L}(\C)$ and the set of $\omega_{FS}$-psh functions which is given by the natural identification
\begin{equation}\label{LL}u\in \mathcal{L}(\C) \to \varphi(z):=
\begin{cases}
u(z)-\frac{1}{2}\log(1+\|z\|^2) & \text{for}\ z\in\C\\
\limsup_{w\in\C\to z}  u(w)-\frac{1}{2}\log(1+\|w\|^2) & \text{for}\ z\in H_{\infty}
\end{cases}  
 \end{equation}  where $\Bbb{P}^m=\C\cup H_{\infty}$ and $H_{\infty}$ denotes the hyperplane at infinity.
Note that since $\Bbb{P}^m$ is compact there is no global psh functions other than the constant ones. On the other hand, we can associate each $\omega_{FS}$-psh function $\varphi$ to its curvature current $\omega_{FS}+dd^c\varphi$ which yields compactness properties of $\omega_{FS}$-psh functions. We use the later properties quite often in this section.  In addition, working in the compact setting makes the usage of integration by parts more simple since there is no boundary.\\ \indent
  We denote the hyperplane bundle $L\to \Bbb{P}^m$ by $L:=\mathcal{O}(1)$ which is endowed with the Fubini-Study metric $h_{FS}$ In the sequel, we identify $\C$ with the affine piece in $\Bbb{P}^m.$ Then the elements of $H^0(\Bbb{P}^m,\mathcal{O}(N))$ can be identified with the homogenous polynomials in $m+1$ variables of degree $N.$ Thus, restricting them to $\Bbb{C}^m,$ we may identify $H^0(\Bbb{P}^m,\mathcal{O}(N))$ with the space of polynomials $Poly(N\Sigma)$ of total degree at most $N$ and the smooth metric $h_{FS}$ can be represented by the weight function $\frac12\log(1+\|z\|^2)$ on $\Bbb{C}^m.$ For each $s_N\in H^0(\Bbb{P}^m,\mathcal{O}(N))$ we let $\|s_N(z)\|_{Nh_{FS}}$ denote the point-wise norm of $s_N$ evaluated with respect to the metric $h_{FS}.$ Then by (\ref{LL}) for each $f_N\in Poly(N\Sigma)$ the function $\frac1N\log|f_N|$ can be naturally identified with $\frac1N\log\|s_N\|_{Nh_{FS}}.$

For $P\subset\R_{\geq0}$ denoting $p=\max\{p_1+\dots+p_m: (p_1,\dots,p_m)\in P\}$ (so that $P\subset p\Sigma$), we may identify $Poly(NP)$ with a subspace of $H^0(\Bbb{P}^m,\mathcal{O}(pN))$ and denote it by $\Pi_{NP}.$  The BM measure $\tau$ induce in inner product on the space $H^0(\Bbb{P}^m,\mathcal{O}(pN))$ defined by $$\|s_N\|^2:=\int_K\|s_N(z)\|^2_{pNh_{FS}}d\tau(z). $$ For a fixed ONB $\{S_j^N\}$ we also let 
$$\mathcal{S}_N(z,z)=\sum_{j=1}^{d_N}\|S^N_j(z)\|^2_{Nh_{FS}}$$ denote restriction of the associated Bergman kernel to diagonal. We remark that the Bergman kernel asymptotics generalize the current setting (see \cite[Proposition 2.9]{B6}). We can endow $\Pi_{NP}$ with $d_N$-fold product measure $\sigma_N$ and we endow the product space $\mathscr{P}=\prod_{N=1}^{\infty}\Pi_{NP}$   with the product measure $\mathcal{P}_{\infty}.$ Note that the elements of the probability space $(\mathscr{P}, \mathcal{P}_{\infty})$ are sequences of random holomorphic sections.
For each $s_N\in\Pi_{NP}$ denoting its zero divisor by $Z_{s_N},$ it follows form Poincar\'e-Lelong formula that $$[Z_{s_N}]=pN\omega_{FS}+dd^c\log\|s_N\|_{pNh_{FS}}.$$  
We remark that $[Z_{s_N}]$ coincides with the (unique) extension of the current of integration $dd^c\log|f_N|$ through the hyperplane at infinity $H_{\infty}.$ Finally,  by (\ref{LL}) the function $V_{P,K,q}$ also extends to a $p\omega_{FS}$-psh function on $\Bbb{P}^m$ which we denote by $V_{P,p\omega_{FS}}$ and define its curvature current by
$$T_{P,K,q}:=p\omega_{FS}+dd^cV_{P,p\omega_{FS}}.$$
  {\bf{Slicing and regularization of currents:}}
  Let $Y$ be a complex manifold of dimension $n$ and $\pi_Y:Y\times\Pm \to Y, \pi_{\Pm}:Y\times \Pm\to \Pm$ denote the projections onto the factors. Given a positive closed $(k,k)$ current $\U$ on $Y\times \Pm$ it follows from \cite{Federer} (se also \cite{DSgeom}) that the slices $\U_y:=\la\U,\pi_Y,y\ra$ exist for a.e. $y\in Y.$ The currents $\U_y$ (if it exists) is a positive closed $(k,k)$ current on $\{y\}\times\Pm.$ For instance, if $\U$ is a continuous form then $\U_y$ is just restriction of $\U$ on $\{y\}\times\Pm.$ We can identify $\U_y$ with a positive closed $(k,k)$ current on $\Pm$ whose mass is independent of $y$ \cite[Lemma 2.4.1]{DS11}.    \\ \indent
  Following \cite{DS11}, we say that the map $y\to \U_y$ defines a \textit{structural variety} in the set of positive closed $(k,k)$ currents on $\Pm.$ We also say that a structural variety is \textit{special} if the slice $\U_y$ exists for every $y\in Y$ and the map $y\to \U_y$ is continuous with respect to weak topology of currents. In this work, we will focus on the following special structural disc: Consider the holomorphic map
  $$H:Aut(\Pm)\times\Pm\to \Pm$$ defined by $H(\tau,z)=\tau^{-1}(z).$ Given a positive closed $(k,k)$ current $R$ on $\Pm$ we define $\U:=H^*(R).$ Then it is easy to see that the slice $\U_{\tau}=\tau_*(R)$ for each $\tau\in Aut(\Bbb{P}^m)$. This in particular implies that $\tau\to R_{\tau}$ is continuous and $\{R_{\tau}\}_{\tau}$ defines a special structural variety \cite[Proposition 2.5.1]{DS11}. \\ \indent
  We let $\Delta\subset \Bbb{C}$ denote the unit disc. We fix a holomorphic chart $Y$ for $Aut(\Pm)$ and denote the local holomorphic coordinates by $y$  where $\|y\|<1$ and $y=0$ corresponds to the identity map $id\in Aut(\Pm).$  We also let $\tau_y\in Aut(\Pm)$ denote the automorphism that correspond to local coordinate $y.$ Next, we fix a positive smooth function $\psi$ with compact support in $\{\|y\|<1\}$ such that $\int\psi(y) dy=1$ and define $\psi_{\theta}(y):=|\theta|^{-2n}\psi(\frac{y}{|\theta|})$ for $\theta\in\Delta.$ Note that $\psi_{\theta}(y)dy$ is an approximate identity for the Dirac mass at 0. Finally, we define the current $ \U\wedge  \pi_{Y}^*(\psi_{\theta}dy)$ by
\begin{eqnarray*} \la \U\wedge  \pi_{Y}^*(\psi_{\theta}dy),\Psi\ra: &=&\int \la \U_y,\Psi\ra \psi_{\theta}(y)dy\\
&=& \int \la \U_{\tau_{\theta y}},\Psi\ra \psi(y)dy
\end{eqnarray*}
  where $\Psi$ is a $(m-k,m-k)$ test form on $Y\times \Pm$
  Note that the slice of $\U\wedge  \pi_{Y}^*(\psi_{\theta}dy)$ can be identified with the current $R_{\theta}$ whose action on the $(m-k,m-k)$ test form $\Theta$ on $\Pm$ defined by 
  $$\la R_{\theta},\Theta\ra:=\int \la (\tau_y)_*R,\Theta\ra\psi_{\theta}(y)dy=\int \la (\tau_{\theta y})_*R,\Theta\ra \psi(y)dy $$ 
 by setting $\Psi=\pi_{\Pm}^*(\Theta).$  
 \begin{prop}\label{shc}
 Let $R$ be a positive closed $(k,k)$ current on $\Pm$ and $\Theta$ is a smooth $(m-k,m-k)$ form on $\Pm$ such that $dd^c\Theta\geq 0.$ Then
 \begin{itemize}
 \item[(i)] $R_{\theta}$ is a smooth positive $(k,k)$ form for $\theta\in\Delta^*.$ The current $R_{\theta}$ depends continuously on $R.$ Moreover, $R_{\theta}\to R$ weakly as $\theta \to 0.$
 \item[(ii)] There exists $C>0$ such that $|\la R_{\theta},\Theta\ra|\leq C\|\Theta\|_{\infty}\|R\|$ for every $\theta\in\Delta.$
 \item[(iii)] $\varphi(\theta):= \la R_{\theta},\Theta\ra$ is a continuous subharmonic function on $\Delta.$ 

 \end{itemize}
 \end{prop}
 \begin{proof}
 Part (i) is proved in \cite[Proposition 2.1.6]{DS11}. Adding a large multiple of $\omega_{FS}$ to $\Theta$ we may assume that $0\leq \Theta\leq C\omega_{FS}$ for some $C>0.$ Since each $R_{\theta}$ is positive closed and its mass is independent of $\theta,$ (ii) follows. For part (iii) let $\Psi:=\pi_{\Pm}^*(\Theta)$ and observe that $\Phi=(\pi_Y)_*(\U\wedge \Psi) $ is of bidegree $(0,0)$ on $Y$ satisfying 
 $$dd^c\Phi=(\pi_Y)_* (\U\wedge dd^c\Psi) \geq 0. $$ This implies that $\Phi$ coincides with a psh function on $Y.$ Note that for fixed $y\in Y$ we have $\varphi(\theta)=\Phi(\theta y)$ for $\theta\in \Delta$ thus $\varphi$ is subharmonic. Continuity follows from (i).
 \end{proof}

\begin{proof}[Proof of Theorem \ref{sa}]
   The proof is based on induction. 
   
\par   \textbf{Case $k=1:$} It is enough to show that $\frac1N\log|f_N(z)| \to V_{P,K,q} $ in $L^1_{loc}(\T).$ First observe that for every $\epsilon>0$ by $(A2)$ and Borel-Cantelli lemma there exists a set $\mathscr{A}\subset \mathscr{P}$ of probability one such that for every sequence $\{f_N\}\in \mathscr{A}$ we have 
\begin{eqnarray*}\label{El}
\log|f_N(z)| &=& \log|\la a^N,u^N(z)\ra|+\frac12\log S_N(z,z)\\
&\leq & \epsilon N+ \frac12\log S_N(z,z)
\end{eqnarray*}
which implies that $$(\limsup_{N\to \infty}\frac1N\log|f_N(z)|)^*\leq V_{P,K,q}(z).$$ 
Note that by $(A3)$, Borel-Cantelli lemma and Proposition \ref{BS} for every $z\in \T$ there exists a set $\mathscr{A}_z\subset \mathscr{P}$ of probability one such that for every $\{f_N\}\in \mathscr{A}_z$ 
\begin{equation}\label{Er}
\liminf_{N\to \infty}\frac1N\log|f_N(z)|\geq V_{P,K,q}(z).
\end{equation} Next, we fix a countable dense subset $z_k\in \T$ and define 
$\mathscr{B}:=\mathscr{A}\cap (\cap_{k=1}^{\infty}\mathscr{A}_{z_k}).$ Clearly, $\mathscr{B}$ has probability one. 
To finish the proof let $\{f_N\}\in\mathscr{B}$ and we assume on the contrary that $\frac1N\log|f_N(z)| \not\to V_{P,K,q} $ in $L^1_{loc}(\T).$ Then there exist a subsequence $f_{N_k}$ and open set $U\Subset \T$ such that $\|f_{N_k}-V_{P,K,q}\|_{L^1(U)}>\epsilon.$ Since $V_{P,K,q}$ is locally bounded above so is $\frac1N\log|f_{N_k}|.$ Then by Hartogs Lemma either $\frac1N\log|f_{N_k}|$ converges uniformly to $-\infty$ or it has a subsequence that converges in $L^1(U).$ If the former occurred than there would exists $n_0\in \Bbb{N}$ such that for $N\geq n_0$ and $z\in U$
$$\frac1N\log|f_N(z)|\leq V_{P,K,q}(z).$$
However, this contradicts  (\ref{Er}). Hence, there exists a subsequence such that $\frac{1}{N_k}\log|f_{N_k}| \to v$ in $L^1(U).$ Then by (\ref{El}) we have $v^*$ is psh,
$v^*\leq V_{P,K,q}$ on $U$ and $v^*\not=V_{P,K,q}.$ Since $V_{P,K,q}$ is continuous the set $U':=\{z\in U: v^*(z)<V_{P,K,q}(z)\}$ is an open set. Hence there exists $z_k\in U'$ but this contradicts (\ref{Er}). 

\par \textbf{Case $k>1$:} We assume that the the claim holds for $k-1.$ By Bertini's theorem for generic $f^k_N\in Poly(NP_k)$ their zero loci $Z_{f^k_N}$ are smooth and intersect transversally. In particular, denoting $\textbf{f}_N^k:=(f_N^1,\dots,f_N^k),$ the current of integration$[Z_{\textbf{f}^k_N}]$ has locally finite mass and 
$$[Z_{\textbf{f}^k_N}]=[Z_{f_N^1}]\wedge[Z_{f_N^2,\dots,f_N^{k}}].$$
Let $\Phi$ be a smooth $(m-k,m-k)$ form on $\Pm.$ Writing the test form $\Phi$ as $\Phi=\Phi^+-\Phi^-$ for some smooth forms $\Phi^{\pm}$ where $dd^c\Phi^{\pm}\geq0$ we may and we do assume that $dd^c\Phi\geq 0.$ We also denote by $[Z_{f_N^2,\dots,f_N^k}]_{\theta}$ the $\theta$-regularization of the current of integration $[Z_{f_N^2,\dots,f_N^k}].$ It follows from Proposition \ref{shc} that
$$u_N(\theta):= \frac{1}{N^k}\la [Z_{f^1_N}]\wedge [Z_{f_N^2,\dots,f_N^k}]_{\theta}, \Phi\ra= \frac{1}{N^k}\la  [Z_{f_N^2,\dots,f_N^k}]_{\theta},  [Z_{f^1_N}]\wedge \Phi\ra$$
defines a continuous subharmonic function on $\Delta.$ Moreover, by $(A2)$, Borel-Cantelli lemma and Cauchy-Schwarz inequality we have
\begin{eqnarray*}
u_N(\theta) &=& \frac{1}{N^k} \la  [Z_{f_N^2,\dots,f_N^k}]_{\theta}, pN\omega_{FS} \wedge \Phi\ra+ \frac{1}{N^k} \la  [Z_{f_N^2,\dots,f_N^k}]_{\theta}, \log\|s^1_N\|_{pNh_{FS}} \Phi\ra \\
&\leq & \frac{1}{N^k} \la  [Z_{f_N^2,\dots,f_N^k}]_{\theta}, pN\omega_{FS} \wedge \Phi\ra+\frac{ \epsilon}{N^{k-1}}  \la  [Z_{f_N^2,\dots,f_N^k}]_{\theta}, \Phi\ra + \frac{1}{N^k} \la  [Z_{f_N^2,\dots,f_N^k}]_{\theta}, \log \sqrt{\mathcal{S}_N(z,z)},\Phi\ra.
\end{eqnarray*}
Then by \cite[Proposition 4.2.6]{DS11}, induction hypothesis and uniform convergence of Bergman functions $\mathcal{S}_N(z,z)$ implies that
$$(\limsup_{N\to \infty}u_N(\theta))^*\leq v(\theta):=\la T_{P,K,q}\wedge (T^{k-1}_{P,K,q})_{\theta}, \Phi\ra\ \text{for}\ \theta\in \Delta$$ where $(T^{k-1}_{P,K,q})_{\theta}$ denotes $\theta$-regularization of $T^{k-1}_{P,K,q}.$ 
In particular, 
$$\limsup_{N\to \infty}\la\frac{1}{N^k} [Z_{\textbf{f}^k_N}],\Phi\ra\leq \la T^k_{P,K,q},\Phi\ra.$$ 
On the other hand, $ [Z_{f_N^2,\dots,f_N^k}]_{\theta}$ is a smooth positive current and since $\frac1N[Z_{f_N^1}]\to T_{P,K,q}$ weakly by Proposition \ref{shc} we have 


\begin{equation}\label{son}
\lim_{N\to \infty}u_N(\theta) =v(\theta)\ \text{for every}\ \theta\in \Delta^*.
\end{equation}
We claim that the equality holds on $\Delta.$ Indeed, if not then there exists a subsequence $N_k$ and a subharmonic function $\varphi$ such that $u_{N_k}\to \varphi$ in $L^1_{loc}(\Delta)$ and
$$\varphi(0)=(\limsup_{N_k\to \infty}u_{N_k}(0))^*<v(0).$$ By above argument $\varphi(\theta)\leq v(\theta)$ for $\theta\in \Delta.$ Hence, by continuity of $v$ the set  
$$\mathcal{O}:=\{\theta\in \Delta: \varphi(\theta)<v(\theta)\}$$ is open. But this contradicts (\ref{son}).
 
   \end{proof}
\section{Unbounded case}
In this section, we obtain generalizations of Theorem \ref{main} and \ref{sa} for certain unbounded closed subsets $K\subset \T.$ Throughout this section we assume that $P\subset \R_{\geq0}$ is an integral polytope with non-empty interior. In the sequel we let $p:=\max\{p_1+\dots+p_m: (p_1,\dots,p_m)\in P\}$ so that $P\subset p\Sigma.$  \\ \indent
 A lower semi-continuous function $q:\C\to\Bbb{R}$ for which $\{z\in K: q(z)<\infty\}$ is non-pluripolar, is called \textit{weakly admissible} if there exists $M\in (-\infty,\infty)$ such that
$$\liminf_{z\in K, \|z\|\to\infty}q(z)-\frac{p}{2}\log(1+\|z\|^2) =M.$$
We say  that  $q$ is a \textit{continuous weakly admissible weight} function for $K$ if it is weakly admissible and it extends to a continuous $p\omega_{FS}$-psh function. In particular, $q$ induces of a continuous metric on $\mathcal{O}(p).$ A weighted closed set $(K,q)$ is called \textit{regular weighted closed set} if the global extremal function $V_{P,K,q}$ extends to a continuous $p\omega_{FS}$-psh function on $\Bbb{P}^m.$ 
We remark that if $q$ is a weakly admissible weight function for $K=\T$ then the set of polynomials $Poly(NP)\subset L^2(e^{-2Nq}dV)$ where $dV=h(z)dz$ denotes a probability volume form on $\C$ (eg. $dV=\frac{1}{m!}\omega_{FS}^m$). 
Then Theorem \ref{extremal} carries over to the present setting and we obtain:
\begin{thm}
Let $P\subset \R_{\geq0}$ be an integral polytope with non-empty interior, $(K,q)$ be a regular weighted closed set and $q:\C\to\Bbb{R}$ be weakly admissible continuous weight function. Then
$$V_{P,K,q}=\lim_{N\to\infty}\frac{1}{N}\log\Phi_N$$ locally uniformly on $\T.$
\end{thm}
Next, we fix an ONB $\{F_N^j\}$ for $Poly(NP)$ with respect to the inner product induced from 
$$\la f,g\ra:=\int_{\T} f(z)\overline{g(z)}e^{-2Nq(z)}dV .$$  We also let $S_N(z,w)$ denote the associated Bergman kernel (cf. \cite[\S 1.1]{B6}). We remark that volume form $dV$ satisfies the weighted Berstein-Markov inequality on $\T$ and the argument in \cite{BloomS} (see also \cite[Proposition 4.2]{SZ1}) generalizes to our setting and we obtain: 
\begin{prop}
Let $P\subset\R_{\geq0}$ be an integral polytope with non-empty interior, $(K,q)$ be a regular weighted  closed set and $q:\C\to\Bbb{R}$ be weakly admissible continuous weight function. Then $$\frac{1}{2N}\log S_N(z,z)\to V_{P,K,q}$$ uniformly on compact subsets of $\T.$
\end{prop}

Hence, following the arguments in proofs of Theorem \ref{main} and Theorem \ref{sa} we obtain:
\begin{thm}\label{unbounded}
Let $P_j\subset \R_{\geq0}$ be an integral polytope with non-empty interior, $(K,q_j)$ be a regular weighted  closed set and $q_j:\C\to\Bbb{R}$ be weakly admissible continuous weight function for each $1\leq j\leq k$.  If condition $(A1)$ holds then
$$N^{-k}\Bbb{E}[Z_{f^1_N,\dots,f^k_N}] \to dd^c(V_{P_1,K,q_1})\wedge\dots \wedge dd^c(V_{P_k,K,q_k})$$ weakly as $N\to \infty.$ 

Moreover, if $(A2)$ and $(A3)$ hold then almost surely
$$N^{-k}Z_{f^1_N,\dots,f^k_N} \to dd^c(V_{P_1,K,q_1})\wedge\dots \wedge dd^c(V_{P_k,K,q_k})$$ weakly as $N\to \infty.$ 
\end{thm}   
Next, we provide an example (from \cite{SZ1}) which falls in the framework of Theorem \ref{unbounded}:
\begin{example}\label{SZex}
Let $P\subset \R_{\geq0}$ be an integral polytope with non-empty interior, $K=\T$ and $q(z)=\frac{p}{2}\log(1+\|z\|^2)$ where $p:=\max\{p_1+\dots+p_m: (p_1,\dots,p_m)\in P\}.$ For each $x\in P$ we denote the \textit{normal cone} to $P$ at $x$ by
$$C_x:=\{u\in\R:\la u,x\ra=\varphi_P(u)\}$$ where $\varphi_P$ is the support function of $P.$ Then by \cite[Lemma 4.3]{SZ1} for every $z\in \T$ there exists unique $\tau_z\in \R$ and $r(z)\in P$ such that 
$$\mu_p(e^{-\frac{\tau_z}{2}}\cdot z)=r(z)\ \text{and}\ \tau_z\in C_{r(z)}$$
where $x\cdot z:=(x_1z_1,\dots,x_mz_m)$ denotes $\R_+$ action on $\T$ and $\mu_P$ denotes the moment map defined in the introduction. Furthermore, by \cite[Theorem 4.1]{SZ1}
 \[
V_{P, p\omega_{FS}} (z)= \begin{cases}
0 & \text{for}\ z\in \mathcal{A}_P\\
\frac12\la r(z),\tau_z\ra-\frac{p}{2}\log[\frac{1+\|z\|^2}{1+\|e^{-\frac{\tau_z}{2}}\cdot z\|^2}] & \text{for}\ z\in\T\backslash \mathcal{A}_P
\end{cases}
\]
extends as a continuous $p\omega_{FS}$-psh function on $\Bbb{P}^m.$ In particular, the weighted global extremal function is given by

\begin{equation}\label{cform}
V_{P,q}(z) = \begin{cases}
\frac{p}{2}\log(1+\|z\|^2) & \text{for}\ z\in \mathcal{A}_P\\
\frac12\la r(z),\tau_z\ra+\frac{p}{2}\log[1+\|e^{-\frac{\tau_z}{2}}\cdot z\|^2] & \text{for}\ z\in\T\backslash \mathcal{A}_P.
\end{cases}
\end{equation}
Letting 
\begin{eqnarray*}
\la f ,g\ra: &=& \int_{\T} f(z)\overline{g(z)}e^{-2Nq(z)}\omega_{FS}^m
\end{eqnarray*}
we see that $$c_Jz^J:=(\frac{(N+m)!}{m!(N-|J|)!j_1!\dots j_m!})^{\frac12}z_1^{j_1}\dots z_m^{j_m} \ for\ J\in NP$$
(where $|J|=j_1+\dots+j_m$) form an ONB for $Poly(NP)$ and a random Laurent polynomial in this context is of the form
$$f_N(z)=\sum_{J\in NP}a_Jc_Jz^J.$$ Thus, Theorem \ref{unbounded} applies (with $P=P_1=P_2$) and almost surely
$$N^{-m}\sum_{\zeta\in Z_{f^1_N,\dots,f^m_N}}\delta_{\zeta}\to MA_{\Bbb{C}}(V_{P,q})\ \text{weakly as}\ N\to \infty $$ 
\end{example}

 \bibliographystyle{alpha}
\bibliography{biblio}
\end{document}